\documentclass[final,1p,times]{elsarticle}

\usepackage{epsfig}

\usepackage{amsmath,amsthm,verbatim,amssymb,amsfonts,amscd, graphicx}
\usepackage{latexsym,graphicx, epsfig, bm, epstopdf, float, color}
\usepackage{tikz,siunitx}
\usepackage{scalerel}
\usetikzlibrary{shapes,arrows,patterns}
\usetikzlibrary{shapes.geometric}
\usepackage{pgfplots}
\usetikzlibrary{shapes.geometric,shapes.symbols}
\usetikzlibrary{patterns}
\usetikzlibrary{tikzmark}
\usepackage{algorithm}
\usepackage[noend]{algpseudocode}
\usepackage{soul}
\usepackage{enumitem}

\usepackage{caption}
\usepackage{subcaption}
\usepackage{graphics,mathtools,hyperref,cleveref}
\usepackage{dsfont}
\usepackage{framed}
\usepackage[most]{tcolorbox}
\usepackage{xcolor}
\usepackage[normalem]{ulem}

\newcommand{\av}[1]{\{\!\!\{#1\}\!\!\}} 
\newcommand{\jump}[1]{\lbrack\!\lbrack #1 \rbrack\!\rbrack} %usata

\renewcommand{\P}{\mathbf{P}}

\newcommand{\x}{\textbf{x}}
\newcommand{\y}{\textbf{y}}

\newcommand{\Tau}{\mathcal{T}}
\newcommand{\vertiii}[1]{{\left\vert\kern-0.25ex\left\vert\kern-0.25ex\left\vert #1 
    \right\vert\kern-0.25ex\right\vert\kern-0.25ex\right\vert}}

\theoremstyle{remark}
\newtheorem{remark}{Remark}

\newtheorem{lemma}{Lemma}
\newtheorem{theorem}{Theorem}

\usepackage[textsize=footnotesize]{todonotes}

\journal{}

\begin{document}

\begin{frontmatter}


%%  \title{Title\tnoteref{label1}}
%%  \tnotetext[label1]{}
%%  \author{Name\corref{cor1}\fnref{label2}}
%%  \ead{email address}

\title{A Posteriori Error Estimation for Parabolic Equations with Enriched Galerkin Finite Element Methods}


\author[CU]{Hyun-Geun Shin}
\author[FSU]{Yi-Yung Yang} %% Author name
\author[FSU]{Sanghyun Lee} %% Author name


\affiliation[CU]{organization={Department of Mathematical Sciences, University of Texas at El Paso},
            addressline={500 W University Ave, El Paso, Texas, USA}}
\affiliation[FSU]{organization={Department of Mathematics, Florida State University},
            addressline={1017 Academic Way}, 
            city={Tallahassee},
            postcode={32306-4510}, 
            state={FL},
            country={USA}}

%% Abstract
% \begin{abstract}
% This paper introduces a novel a posteriori error estimation framework for the enriched Galerkin (EG) finite element method applied to linear parabolic equations. While the EG method has been recognized for its local conservation property and computational efficiency compared to discontinuous Galerkin methods, its mathematical analysis in the context of a posteriori error estimation for parabolic problems remains unexplored. In this work, we prove reliability and efficiency using the residual-based approach. Furthermore, we integrate these error estimators into an adaptive mesh refinement strategy, demonstrating their effectiveness in achieving efficient and reliable error control through several numerical examples. The proposed approach provides a significant advancement in the mathematical foundation and practical applicability of the EG method for time-dependent problems.
% \end{abstract}
\begin{abstract}{
This paper develops a residual-based a posteriori error estimator for the enriched Galerkin (EG) finite element approximation of linear parabolic equations. The EG method combines a globally continuous finite element space with an elementwise constant enrichment, thereby retaining local conservation with fewer degrees of freedom than a comparable fully discontinuous Galerkin discretization. Using backward Euler time discretization, we analyze the spatial residual of the associated time-discrete elliptic problem at each fixed time level on shape-regular quadrilateral meshes. Reliability and local efficiency estimates are established for the incomplete and nonsymmetric interior penalty formulations, corresponding to $\theta=0$ and $\theta=1$, respectively, whereas the symmetric interior penalty case is not covered by the present analysis. Temporal discretization and mesh-transfer errors are not estimated separately. The resulting estimator is incorporated into an adaptive mesh refinement algorithm, and numerical experiments, including tests on distorted quadrilateral meshes, demonstrate its ability to identify localized singular behavior and reduce the computational cost compared with uniform refinement.}
\end{abstract}



%% Keywords
\begin{keyword}
  enriched Galerkin method
  \sep a posteriori error
  \sep reliability \sep efficiency


\end{keyword}

\end{frontmatter}

%\linenumbers

\section{Introduction}

Numerical errors in computational simulations pose significant challenges, as they directly affect the accuracy, reliability, and efficiency of numerical solutions. Over the years, considerable efforts have been devoted to mitigating such errors, leading to significant advancements in a posteriori error estimation methods~\cite{AINSWORTH19971}. By coupling these estimators with dynamic adaptive algorithms, one can effectively manage computational resources to minimize error and optimize performance.

In this work, we investigate a linear parabolic partial differential equation discretized using the enriched Galerkin (EG) finite element method~\cite{Shuyu2009,Becker2004}. The EG method augments continuous Galerkin finite element spaces with piecewise constant functions, thereby achieving local conservation at a lower computational cost than classical nonconforming discontinuous Galerkin finite element methods~\cite{LEE_EG_Elliptic_Parabolic,LEE201719}. Recently, EG methods have been employed in a variety of multiphysics applications, including two-phase flow in porous media~\cite{lee2018enriched}, poroelasticity~\cite{kadeethum2021enriched,lee2023locking,choo2018enriched}, hyperbolic equations~\cite{kuzmin2025bound}, thermoporoelasticity~\cite{yi2024physics}, and flow problems with fractures~\cite{kadeethum2020flow,lee2016phase}.

% \textcolor{blue}{Despite these developments, a posteriori error analysis for EG methods remains much less developed than the corresponding theory for conforming and discontinuous Galerkin finite element methods. In particular, to the best of our knowledge, there has been no systematic residual-based a posteriori error analysis for EG approximations of parabolic problems based on the residuals of the associated time-discrete elliptic problems. This is the main gap addressed in the present work.}
{Despite these developments, a posteriori error analysis for EG methods remains much less developed than the corresponding theory for conforming and discontinuous Galerkin finite element methods. In particular, to the best of our knowledge, there has been no systematic residual-based a posteriori error analysis for EG approximations of parabolic problems. {The present work addresses this gap by developing residual estimators adapted to the continuous--discontinuous structure of the EG approximation space.} }

Traditionally, residual-based a posteriori error analysis has been well-established for conforming finite element methods, with original developments in the context of linear elliptic problems before extension to parabolic settings~\cite{adjerid1999posteriori, gaevskaya2005posteriori}. For instance, \cite{picasso1998adaptive} considered a normalized heat equation solved via backward Euler and piecewise linear finite elements, establishing reliability through Clement’s interpolant and efficiency using bubble functions. Subsequently, \cite{chen2004adaptive} proposed error estimators and an adaptive algorithm echoing \cite{picasso1998adaptive}, aided by the MNS refinement strategy~\cite{morin2000data}. Such approaches primarily focused on error control in the $L^2(0,T;L^2(\Omega))$ norm.

% {\color{blue}As the field evolved, elliptic reconstruction methods broadened the scope of a posteriori estimates. Initially introduced by \cite{makridakis2003elliptic} and applied in a fully discrete setting by \cite{lakkis2006elliptic}, elliptic reconstruction has since spurred various enhancements. For example, \cite{bansch2012posteriori,bansch2013effect} showed its applicability to both mesh refinement and coarsening, whereas \cite{karakatsani2016posteriori} integrated it into fractional step $\vartheta$-approximations.
% More recent developments include long-time error estimation~\cite{sutton2020long}, analyses of delay PDEs~\cite{wang2022delay}, derivation of lower bounds~\cite{georgoulis2023lower}, and max-norm estimates~\cite{linbeta2024unified}. Further research extended reconstruction ideas to space-time approaches, notably for Crank--Nicolson~\cite{akrivis2006posteriori} and non-conforming discontinuous Galerkin schemes~\cite{kim2008adaptive, gupta2018posteriori, karakatsani2016posteriori}.}
{As the field evolved, elliptic reconstruction methods broadened the scope of a posteriori estimates for parabolic problems. Initially introduced by \cite{makridakis2003elliptic} and applied in a fully discrete setting by \cite{lakkis2006elliptic}, elliptic reconstruction has been used to analyze mesh refinement and coarsening~\cite{bansch2012posteriori,bansch2013effect}, fractional-step approximations~\cite{karakatsani2016posteriori}, Crank--Nicolson discretizations~\cite{akrivis2006posteriori}, and non-conforming discontinuous Galerkin schemes~\cite{kim2008adaptive, gupta2018posteriori, karakatsani2016posteriori}.}
% 

% {\color{orange}
% Meanwhile, a posteriori error estimators have also been generalized to various non-conforming and locally conservative numerical approaches. For instance, \cite{cangiani2021adaptive} investigated non-hierarchical Galerkin techniques, applying them to virtual element methods, while \cite{dai2024posteriori} examined weak Galerkin finite elements. In \cite{ray2023adaptive}, adaptive immersed finite element methods were proposed for interface problems.}

% {\color{blue}
% Meanwhile, a posteriori error estimators have also been generalized to other non-conforming and locally conservative numerical approaches, including virtual element methods~\cite{cangiani2021adaptive}, weak Galerkin finite elements~\cite{dai2024posteriori}, and adaptive immersed finite element methods for interface problems~\cite{ray2023adaptive}.}

For applications in which elementwise local conservation is essential~\cite{lee2018enriched,kadeethum2020flow}, discontinuous Galerkin (DG) methods have emerged as a widely adopted and thoroughly studied class of approaches, offering both rigorous error control and local conservation properties. For instance, \cite{georgoulis2011posteriori} established a posteriori error estimates for a DG method in space using elliptic reconstruction and a decomposition of continuous and discontinuous solution components.
{The purpose of this paper is to derive and test residual-based estimators specifically adapted to the EG finite element space. EG provides a locally conservative alternative to DG while using substantially fewer degrees of freedom than a comparable fully discontinuous approximation. Unlike DG methods, which employ fully discontinuous approximation spaces, the EG method combines a globally continuous finite element component with a discontinuous piecewise constant enrichment. Although the EG formulation employs the same interior-penalty bilinear-form structure as DG, its approximation space is only partially discontinuous. The continuous component produces no interior solution jumps, while the discontinuities arise from the element-wise constant enrichment. Consequently, the approximation arguments and the treatment of the jump contributions in the residual analysis must be adapted to the EG space rather than transferred directly from a fully discontinuous DG analysis. Accordingly, the proposed estimator reflects the continuous--discontinuous structure of EG and} contains the standard cell residual, flux-jump residuals, boundary residuals, and enrichment-jump contributions associated with the element-wise constant enrichment.

% \textcolor{blue}{The parabolic equation is discretized by the EG method in space and by the backward Euler method in time. The analysis in this paper focuses on the spatial residual at each fixed time level. More precisely, at each time step, the backward Euler discretization gives an elliptic problem on the current mesh, and our estimator measures the corresponding spatial residual. Thus, the temporal discretization error and the additional error introduced by transferring the previous solution between different adaptive meshes are not estimated separately in the present work. This distinction is important for the interpretation of the theoretical results and the numerical effectivity indices.}
{The parabolic equation is discretized using the EG method in space and the backward Euler method in time. The analysis in this paper focuses on the spatial residual associated with the time-discrete problem at each fixed time level. More precisely, at each time step, the backward Euler discretization gives an elliptic problem on the current mesh, and our estimator measures the corresponding spatial residual. Therefore, the reliability and efficiency estimates established in this work quantify the spatial discretization error of the time-discrete problem. Thus, the temporal discretization error and the additional error introduced by transferring the previous solution between different adaptive meshes are not estimated separately in the present work. This distinction is important for the interpretation of the theoretical results and the numerical effectivity indices.}

The main contributions of this paper are summarized as follows. First, we formulate the fully discrete EG approximation for linear parabolic equations on shape-regular quadrilateral meshes, allowing the finite element space to vary in time under adaptive mesh refinement. When the mesh changes between two consecutive time levels, the previous numerical solution is transferred to the current finite element space before evaluating the backward Euler residual. Second, we derive a residual-based a posteriori estimator for the EG approximation and prove reliability and local efficiency estimates for the spatial residual associated with the time-discrete elliptic problem obtained from the backward Euler discretization. {The analysis applies to the incomplete and nonsymmetric interior penalty formulations, corresponding to $\theta=0$ and $\theta=1$, respectively; the symmetric interior penalty case $\theta=-1$ is not covered by the present proof.} Third, we incorporate the estimator into an adaptive mesh refinement algorithm and demonstrate its performance through numerical examples involving singular solutions on an L-shaped domain, including tests on distorted quadrilateral meshes. The numerical results show that the proposed estimator successfully identifies localized singular behavior and substantially reduces the number of degrees of freedom compared with uniform refinement.

The remainder of this paper is organized as follows. In Section~\ref{sec:2}, we present the mathematical model for the given parabolic equation, and the temporal and spatial discretization including the enriched Galerkin methods are discussed in Section~\ref{sec:3}.
The a posteriori error estimates are derived in Section~\ref{sec:4}, including the reliability and efficiency results.
Finally, in Section~\ref{sec:num}, we present the results of numerical experiments designed to verify the derived theory and illustrate the capabilities and performance of the proposed adaptive algorithm.

\section{Governing system}
\label{sec:2}

We consider a parabolic diffusion model motivated by pressure flow in porous media.
Given an open, bounded domain $\Omega\subset\mathbb R^d$, $d=2,3$, and a final time $T>0$, we consider
\begin{equation}\label{eqn:main}
\partial_t p + \nabla \cdot (-K \nabla p) = f
\quad
\text{in } \Omega \times(0,T).
\end{equation}
Here, $p:=p(\mathbf x,t)\in\mathbb R$ is the primary scalar unknown.
In the porous-media context, $p$ represents the pressure at spatial location $\mathbf x\in\Omega$ and time $t\in(0,T)$.

We also define the Darcy velocity by
$\mathbf u := -K\nabla p.$
In this paper, we restrict our discussion to the isotropic permeability case
$K(\mathbf x)=\kappa(\mathbf x)\mathbf I$,
where $\mathbf I \in \mathbb{R}^{d\times d}$ is the identity matrix and $\kappa(\mathbf x)$ is a positive scalar permeability field. We assume that there exist constants $0<\kappa_{\min}\leq \kappa_{\max}<\infty$ such that

\begin{equation}
\kappa_{\min}
\leq
\kappa(\mathbf x)
\leq
\kappa_{\max}
\qquad
\text{for a.e. } \mathbf x\in\Omega.
\end{equation}
In addition, we assume that $\kappa$ is sufficiently regular on each mesh element, for instance $\kappa|_T\in W^{1,\infty}(T)$, so that the element residual $\nabla\cdot(K\nabla p_h)$ appearing in the a posteriori estimator is well-defined.

The equation is supplemented by a source term $f:=f(\mathbf x,t)\in\mathbb R$ and the following boundary and initial conditions:
\begin{equation}
p = g_D
\quad \text{on } \partial \Omega_D \times (0,T),
\qquad
-K \nabla p \cdot \mathbf n = g_N
\quad \text{on } \partial \Omega_N \times(0,T),
\qquad
p(\mathbf{x},0) = p_0
\quad \text{in } \Omega.
\label{eqn:bc}
\end{equation}
Here, $\partial \Omega_D\cup\partial \Omega_N = \partial \Omega$ and
$\partial \Omega_D\cap \partial \Omega_N = \emptyset$, representing a partition of the boundary $\partial\Omega$ into Dirichlet and Neumann parts.
The functions $g_D$ and $g_N$ denote the prescribed Dirichlet and Neumann data, respectively, and $\mathbf n:=\mathbf n(\mathbf x)\in\mathbb R^d$ is the unit outward normal vector to $\partial\Omega$.
Finally, $p_0:=p_0(\mathbf x)$ denotes the initial condition.
%\textcolor{blue}{With the sign convention in \eqref{eqn:bc}, the Neumann datum $g_N$ represents the outward Darcy flux}
%$$
%\textcolor{blue}{
%\mathbf u\cdot \mathbf n = -K\nabla p\cdot \mathbf n = g_N.}
%$$
%\textcolor{blue}{Equivalently, the exact solution satisfies $g_N+K\nabla p\cdot \mathbf n=0$ on $\partial\Omega_N\times(0,T)$, which is the sign convention used later in the Neumann boundary residual.}

\section{Numerical discretization}
\label{sec:3}

\subsection{Enriched Galerkin finite element spaces}

Let $\Tau_h$ be a shape-regular partition of $\Omega$ into quadrilateral elements $T$ in the sense of Ciarlet. That is, there exists a constant $\sigma>0$, independent of the mesh size, such that
$h_T/\rho_T \leq \sigma$ for all $T\in \Tau_h$,
where $h_T$ is the diameter of $T$ and $\rho_T$ is the diameter of the largest inscribed ball of $T$.
We also assume the standard local comparability between element and edge sizes on shape-regular quadrilateral meshes. Namely, for each edge $\gamma\subset\partial T$, the quantities $h_T$ and $h_\gamma$ are uniformly comparable, with constants depending only on the shape-regularity parameter.

The set of all edges $\gamma$ in $\Tau_h$ is denoted by
$\varepsilon_h=
\varepsilon_h^I
\cup
\varepsilon_h^N
\cup
\varepsilon_h^D$,
where $\varepsilon_h^I$ is the set of all interior edges, and $\varepsilon_h^D$ and $\varepsilon_h^N$ are the sets of boundary edges on which Dirichlet and Neumann conditions are imposed, respectively. These three sets are mutually disjoint. For convenience, we also define
$
\varepsilon_h^{ID}
:=
\varepsilon_h^I
\cup
\varepsilon_h^D
$.

We now introduce the finite element spaces. Let $\mathbb Q_k(T)$ denote the tensor-product polynomial space of degree at most $k$ in each coordinate direction on an element $T$; see, e.g., \cite{ciarlet1978finite,riviere2008discontinuous}. Throughout this paper, we assume $k\geq 1$.
The continuous Galerkin space based on $\mathbb Q_k$ Lagrange elements is defined by
\begin{equation}
M_0^k(\Tau_h)
:=
\left\{
v\in L^2(\Omega)
:
v|_T\in \mathbb Q_k(T)
\ \forall T\in\Tau_h
\right\}
\cap C^0(\overline{\Omega}).
\end{equation}
Here, the subscript in $M_0^k(\Tau_h)$ is used only to distinguish the continuous component from the discontinuous space below; it does not indicate homogeneous boundary conditions.
Similarly, the discontinuous $\mathbb Q_k$ finite element space is defined by
\begin{equation}
M^k(\Tau_h)
:=
\left\{
v\in L^2(\Omega)
:
v|_T\in \mathbb Q_k(T)
\ \forall T\in\Tau_h
\right\}.
\end{equation}
The enriched Galerkin finite element space of order $k$ is then given by
\begin{equation}
V_{h,k}^{\mathrm{EG}}
=
M_0^k(\Tau_h)
+
M^0(\Tau_h).
\end{equation}
Thus, $V_{h,k}^{\mathrm{EG}}$ consists of a globally continuous $\mathbb Q_k$ component enriched by element-wise constants. This is the key distinction between the EG space and the fully discontinuous Galerkin space $M^k(\Tau_h)$.

Since functions in $V_{h,k}^{\mathrm{EG}}$ are polynomial on each element, we have
\begin{equation}
V_{h,k}^{\mathrm{EG}} \subset H^2(\Tau_h).
\end{equation}
Here,
\begin{equation}
H^2(\Tau_h)
=
H^2(\Omega,\Tau_h)
:=
\left\{
v\in L^2(\Omega)
:
v|_T\in H^2(T),
\ \forall T\in \Tau_h
\right\}.
\end{equation}
This space is equipped with the broken Sobolev norm
\begin{equation}
\vertiii{v}_{H^2}
:=
\left(
\sum_{T\in\Tau_h}
\|v\|_{H^2(T)}^2
\right)^{1/2}.
\end{equation}

Since functions in $H^2(\Tau_h)$ may be discontinuous across interior edges, we define jumps and averages as follows. Let $\gamma\in\varepsilon_h^I$ be an interior edge shared by two neighboring elements $T_-$ and $T_+$, and let $\mathbf n_\gamma$ be a fixed unit normal vector on $\gamma$ oriented from $T_-$ to $T_+$. Then, for a scalar function $v\in H^2(\Tau_h)$, we define
\begin{equation}
\jump{v}
=
\jump{v}_\gamma
:=
\left(v|_{T_-}\right)\big|_\gamma
-
\left(v|_{T_+}\right)\big|_\gamma,
\end{equation}
and
\begin{equation}
\av{v}
=
\av{v}_\gamma
:=
\frac{1}{2}
\left(v|_{T_-}\right)\big|_\gamma
+
\frac{1}{2}
\left(v|_{T_+}\right)\big|_\gamma .
\end{equation}
On boundary edges $\gamma\in\varepsilon_h^D\cup\varepsilon_h^N$, we set
\begin{equation}
\jump{v}
=
\av{v}
=
v|_\gamma .
\end{equation}

For vector-valued quantities, such as $K\nabla v$, the jump and average are understood componentwise. In particular, the normal flux jump appearing later is defined on an interior edge by
\begin{equation}
\jump{\mathbf n_\gamma\cdot K\nabla v}
:=
\left(\mathbf n_\gamma\cdot K\nabla v|_{T_-}\right)\big|_\gamma
-
\left(\mathbf n_\gamma\cdot K\nabla v|_{T_+}\right)\big|_\gamma .
\end{equation}

\subsection{Spatial discretization using enriched Galerkin method}

First, we provide the spatial discretization of the system \eqref{eqn:main}--\eqref{eqn:bc} using the EG method.
The semi-discrete formulation is given as follows: find
$p_h\in L^2(0,T;V^{\mathrm{EG}}_{h,k})$ such that
\begin{equation}\label{eq:semi-EG-formulation}
\left(\partial_t p_h, w_h\right) + A_\theta(p_h(t), w_h) = F_\theta\left(w_h;t\right),
\qquad\forall w_h\in V^{\mathrm{EG}}_{h,k},
\qquad\forall t\in(0,T].
\end{equation}
Here, $A_\theta(\cdot,\cdot):H^2(\Tau_h)\times H^2(\Tau_h)\to\mathbb R$ is the bilinear form defined by
\begin{equation}
\begin{split}
A_\theta(p_h(t), w_h)
&:=
\sum_{T\in\Tau_h}
\left(K\nabla p_h(t), \nabla w_h\right)_T
-
\sum_{\gamma \in \varepsilon^{ID}_h}
\left(
\av{\mathbf n_\gamma\cdot K\nabla p_h(t)},
\jump{w_h}
\right)_\gamma
\\
&\quad
+
\theta
\sum_{\gamma \in \varepsilon^{ID}_h}
\left(
\jump{p_h(t)},
\av{\mathbf n_\gamma\cdot K\nabla w_h}
\right)_\gamma
+
\alpha
\sum_{\gamma\in\varepsilon^{ID}_h}
h_\gamma^{-1}K_{\max}(\gamma)
\left(
\jump{p_h(t)},
\jump{w_h}
\right)_\gamma .
\end{split}
\label{eq:A_theta}
\end{equation}
The linear functional $F_\theta(w_h;t)$ is given by
\begin{equation}
\begin{split}
F_\theta(w_h;t)
&:=
\sum_{T\in\Tau_h}
\left(f(t), w_h\right)_T
-
\sum_{\gamma \in \varepsilon^N_h}
\left(g_N(t), w_h\right)_\gamma
\\
&\quad
+
\theta
\sum_{\gamma\in \varepsilon_h^D}
\left(
g_D(t),
\mathbf n_\gamma\cdot K\nabla w_h
\right)_\gamma
+
\alpha
\sum_{\gamma\in \varepsilon_h^D}
h_\gamma^{-1}K_{\max}(\gamma)
\left(
g_D(t), w_h
\right)_\gamma .
\end{split}
\label{eq:F_theta}
\end{equation}
Here, $\alpha=\alpha(k)>0$ is the penalty parameter depending on the polynomial degree $k$.
Since we restrict our discussion to the isotropic permeability case $K(\mathbf x)=\kappa(\mathbf x)\mathbf I$, we define
\begin{equation}
K_{\max}(\gamma)
:=
\max_{\mathbf x\in\gamma}\kappa(\mathbf x).
\end{equation}

The different choices of $\theta$ lead to the symmetric interior penalty Galerkin formulation (SIPG) with $\theta=-1$ \cite{SIPG}, the incomplete interior penalty Galerkin formulation (IIPG) with $\theta=0$ \cite{DGscheme}, and the nonsymmetric interior penalty Galerkin formulation (NIPG) with $\theta=1$ \cite{sun2005symmetric,NIPG}. The reliability and efficiency analyses presented in this paper apply to the IIPG and NIPG formulations, corresponding to $\theta\geq 0$; the SIPG case is not covered by the present proof.

Note that the bilinear form in the EG formulation \eqref{eq:semi-EG-formulation} has the same structure as the corresponding DG formulation \cite{riviere2008discontinuous}. The key distinction is the choice of finite element space: the DG method uses a fully discontinuous space, whereas the EG method uses a continuous $\mathbb Q_k$ Lagrange component enriched by element-wise constants. Therefore, the consistency argument for the EG formulation follows from the DG consistency argument at the level of the bilinear form. In particular, if $p\in H^1(0,T;H^2(\Tau_h))$ is the solution to \eqref{eqn:main}--\eqref{eqn:bc}, then $p$ also satisfies \eqref{eq:semi-EG-formulation}; see Lemma~4.1 in \cite{riviere2008discontinuous}.

\subsection{Fully discrete enriched Galerkin formulation}

To consider temporal discretization of \eqref{eq:semi-EG-formulation}, we introduce a discrete time partition
$0=t_0<\cdots<t_n<\cdots<t_N=T$
with a uniform time step $\delta t=T/N$, where $N\in\mathbb N$.
Including the case that the mesh may change over time, we denote the triangulation, edge set, and finite element space at time $t_n$ by
$\Tau_{h,n}$, $\varepsilon_{h,n}$, and $V^{\mathrm{EG}}_{h,k,n}$, respectively, throughout this paper.

To distinguish from the semi-discrete formulation \eqref{eq:semi-EG-formulation}, we also denote
$g_D^n=g_D(t_n)$, $g_N^n=g_N(t_n)$, and $f^n=f(t_n)$.

When the mesh changes between two consecutive time levels, the numerical solution from the previous time level must be transferred to the current finite element space before the backward Euler term is evaluated. Therefore, we introduce a mesh-transfer operator
\begin{equation}
\mathcal I_n:
V^{\mathrm{EG}}_{h,k,n-1}
\to
V^{\mathrm{EG}}_{h,k,n}.
\end{equation}
For example, $\mathcal I_n$ may be chosen as an $L^2$-projection or a stable interpolation/projection operator. If the mesh does not change between two consecutive time steps, then $\mathcal I_n$ is simply the identity operator.

Then, the fully discrete EG formulation of problem \eqref{eqn:main}--\eqref{eqn:bc}, advanced in time from $t=t_{n-1}$ to $t=t_n$ with the backward Euler time discretization scheme, reads as follows: given
$p_h^{n-1}\in V^{\mathrm{EG}}_{h,k,n-1}$, find
$p_h^n\in V^{\mathrm{EG}}_{h,k,n}$ such that
\begin{equation}\label{eq:fullly_EG_discretization}
S_\theta(p_h^n,w_h)
=
F_\theta^n(w_h),
\qquad
\forall w_h\in V^{\mathrm{EG}}_{h,k,n},
\end{equation}
where
\begin{align}
S_\theta(p_h^n,w_h)
&:=
\frac{1}{\delta t}
(p_h^n,w_h)
+
A_\theta(p_h^n,w_h),
\label{eq:S_theta}
\\
F_\theta^n(w_h)
&:=
\frac{1}{\delta t}
(\mathcal I_n p_h^{n-1},w_h)
+
F_\theta(w_h;t_n).
\label{eq:F_theta_n}
\end{align}
Here, $p_h^n$ is the approximation to $p_h(\cdot,t_n)$ in \eqref{eq:semi-EG-formulation}.

Equivalently, the time-discrete term in \eqref{eq:fullly_EG_discretization} is interpreted as
\begin{equation}
\frac{1}{\delta t}
\left(
p_h^n-\mathcal I_n p_h^{n-1},
w_h
\right),
\qquad
w_h\in V^{\mathrm{EG}}_{h,k,n}.
\end{equation}
Thus, all inner products in the fully discrete formulation are evaluated on the current mesh $\Tau_{h,n}$ at time level $n$. The transfer may introduce the standard projection/interpolation error associated with mesh adaptation; this error is not analyzed separately in the present work.

\begin{remark}
The transfer operator $\mathcal I_n$ is included to make the fully discrete formulation well-defined on adaptive meshes. In the fixed-mesh case, $\mathcal I_n$ reduces to the identity map and the formulation becomes the standard backward Euler EG method.
\end{remark}

\begin{remark}
In \eqref{eq:fullly_EG_discretization}, \eqref{eq:S_theta}, and \eqref{eq:F_theta_n}, the test function belongs to the finite element space at the current time level, namely $V^{\mathrm{EG}}_{h,k,n}$. Therefore, it could be denoted more explicitly by $w_h^n$. For notational simplicity, however, we write $w_h$ throughout the formulation. The dependence on the time level $n$ is understood from the test space $V^{\mathrm{EG}}_{h,k,n}$.
\end{remark}

%in this paper, we will based on the residual of the %numerical solution $p^n_h\in V^\text{EG}_{h,k,n}$ to \eqref{eq:fullly_EG_discretization}
%{\color{blue}A posteriori error estimates provide an bounds on the error $\|p-p_h\|$ in a chosen norm of interest in the form of $\|p - p_h\|\leq C\eta $. These bounds are commonly referred to as the error estimator $\|p - p_h\|\leq C\eta $. If the error estimator can be computed element-wise on an given triangulation $\Tau_h$, for example, $\eta = \left(\sum_{T\in\Tau_h}\eta^2_T\right)^{1/2}$, where we called $\eta_T$ as local error estimator, then we may use $\{\eta_T\}_{T\in\Tau_h}$  to determine whether elements $T\in \Tau_h$ should be refined to reduce $\|p-p_h\|$. Such upper bounds
%Meanwhile, 
%and in this paper, we will based on the residual of the numerical solution $p^n_h\in V^\text{EG}_{h,k,n}$ to \eqref{eq:fullly_EG_discretization}
% }
%{
%\color{blue}
%}
%In this section, we present the a posteriori error analysis to derive the upper and lower bounds. 

% Our analysis starts with defining a dual norm
% \begin{equation}\label{eq:dual_norm}
%     \|\mathcal R\|_* := \sup_{w\in H^2(\Omega, \Tau_h)} \dfrac{\mathcal R(w)}{\|w\|_{H^2(\Omega, \Tau_{h,n})}},
% \end{equation}
% where 
% $\mathcal R\in \left(H^2(\Omega,\Tau_h)\right)^*$ as 
% \begin{equation}\label{eq:residual}
%     \mathcal R(w) :=  F(w) - S_\theta(p^n_h, w), \quad p^n_h \in V^\text{EG}_{h,k,n}, \quad  1\leq n\leq N.
% \end{equation}

\section{A posteriori error analysis}
\label{sec:4}

A posteriori error estimates use the residual to quantify the error between the numerical solution and the exact solution, thereby providing a practical and computable guide for adaptive mesh refinement.
Since the present work focuses on spatial error estimation at each fixed backward Euler time level, we first clarify the time-discrete problem with respect to which the residual is defined.

% {\color{purple}{
% At time level $t_n$, after transferring the previous numerical solution to the current finite element space, the backward Euler discretization gives an elliptic problem on the current mesh $\Tau_{h,n}$. We define $\widetilde p^n$ to be the exact solution of this time-discrete problem, satisfying
% }
% \begin{equation}
% {
% S_\theta(\widetilde p^n,w)
% =
% F_\theta^n(w),
% \qquad
% \forall w\in H^2(\Omega,\Tau_{h,n}).
% }
% \label{eq:time_discrete_exact_solution}
% \end{equation}
% {
% Here, $F_\theta^n$ contains the transferred previous numerical solution $\mathcal I_n p_h^{n-1}$, as defined in \eqref{eq:F_theta_n}. Therefore, $\widetilde p^n$ is not the exact continuous-in-time solution $p(\cdot,t_n)$, but rather the exact solution of the backward Euler problem at the fixed time level $t_n$. The difference $p(\cdot,t_n)-\widetilde p^n$ represents the temporal consistency error, and, on adaptive meshes, may also include the effect of transferring data between consecutive finite element spaces. These contributions are not estimated separately in the present work.
% }}
{
At time level $t_n$, after transferring the previous numerical solution to the current finite element space, the backward Euler discretization gives an elliptic problem on the current mesh $\Tau_{h,n}$. We define $\widetilde p^n$ to be the exact solution of this time-discrete elliptic problem, satisfying
\begin{equation}
    S_\theta(\widetilde p^n,w)
    =
    F_\theta^n(w)
    =
    \frac{1}{\delta t}\left(I_n p_h^{n-1},w\right)
    +
    F_\theta(w;t_n),
    \qquad
    \forall w\in H^2(\Omega,\Tau_{h,n}).
    \label{eq:time_discrete_exact_solution}
\end{equation}
Thus, $\widetilde p^n$ is the exact solution of the current time-discrete elliptic problem with the transferred previous numerical solution $I_n p_h^{n-1}$ treated as known data. It is not the exact continuous-in-time solution $p(\cdot,t_n)$. Consequently, the difference $p(\cdot,t_n)-\widetilde p^n$ may contain the temporal consistency error, the error inherited from the previous numerical solution, and, on adaptive meshes, the effect of transferring data between consecutive finite element spaces. These contributions are not estimated separately in the present work.
}

The residual of the numerical solution $p^n_h\in V^{\mathrm{EG}}_{h,k,n}$ to \eqref{eq:fullly_EG_discretization} is defined by
\begin{equation}
\label{eq:residual}
\mathcal R^n(w)
:=
F^n_\theta(w)
-
S_\theta(p^n_h,w),
\qquad
\forall w\in H^2(\Omega,\Tau_{h,n}).
\end{equation}
It belongs to the dual space $\left(H^2(\Omega,\Tau_{h,n})\right)^*$ associated with the dual norm
\begin{equation}
\label{eq:dual_norm}
\|\mathcal R^n\|_*
:=
\sup_{0\neq w\in H^2(\Tau_{h,n})}
\frac{|\mathcal R^n(w)|}{\|w\|_{H^2(\Tau_{h,n})}}.
\end{equation}
By the definition of $\widetilde p^n$ in \eqref{eq:time_discrete_exact_solution}, we have
\begin{equation}
\mathcal R^n(w)
=
S_\theta(\widetilde p^n,w)
-
S_\theta(p_h^n,w)
=
S_\theta(\widetilde e^n,w),
\end{equation}
where
\begin{equation}
\widetilde e^n
:=
\widetilde p^n-p_h^n
\end{equation}
denotes the spatial error with respect to the exact solution of the time-discrete problem.

Furthermore, by the Galerkin orthogonality,
\begin{equation}
S_\theta(\widetilde p^n-p_h^n,w_h)
=
S_\theta(\widetilde e^n,w_h)
=
0,
\qquad
\forall w_h\in V^{\mathrm{EG}}_{h,k,n},
\label{eq:orthogonal-error-femspace}
\end{equation}
we have, for any $w_h\in V^{\mathrm{EG}}_{h,k,n}$,
\begin{equation}
\label{eq:orthogonality}
\mathcal R^n(w)
=
S_\theta(\widetilde e^n,w)
=
S_\theta(\widetilde e^n,w-w_h)
=
\mathcal R^n(w-w_h).
\end{equation}
Consequently,
\begin{equation}
\|\mathcal R^n\|_*
=
\sup_{0\neq w\in H^2(\Tau_{h,n})}
\frac{|\mathcal R^n(w-w_h)|}{\|w\|_{H^2(\Tau_{h,n})}},
\end{equation}
where $w_h\in V^{\mathrm{EG}}_{h,k,n}$ is chosen as a suitable finite element approximation of $w$.

In the following subsections, we derive upper and lower bounds for this spatial residual. The reliability estimate provides a computable upper bound for $\|\mathcal R^n\|_*$ in terms of local residual indicators, while the local efficiency estimate shows that the indicators are bounded below by the residual up to local data-oscillation terms. These results should be interpreted as spatial residual estimates for the time-discrete problem at each fixed time level.

\begin{remark}
The residual $\mathcal R^n(w)$ is defined at the fixed time level $t_n$ and is used only for the spatial a posteriori error estimator. It is not the residual associated with a piecewise linear reconstruction in time between $p_h^{n-1}$ and $p_h^n$. All terms in $\mathcal R^n(w)$ are evaluated on the current mesh $\Tau_{h,n}$. If $\Tau_{h,n-1}\neq \Tau_{h,n}$, then $p_h^{n-1}$ is first transferred to the current finite element space by the mesh-transfer operator
$\mathcal I_n:V^{\mathrm{EG}}_{h,k,n-1}\to V^{\mathrm{EG}}_{h,k,n}$, and the time-discrete term is interpreted as
\[
\frac{1}{\delta t}
\left(
p_h^n-\mathcal I_n p_h^{n-1},
w
\right).
\]
Thus, no additional temporal reconstruction is introduced in the present analysis.
\end{remark}

\begin{remark}
The full error with respect to the exact continuous-in-time solution can be decomposed as
\[
p(\cdot,t_n)-p_h^n
=
\bigl(p(\cdot,t_n)-\widetilde p^n\bigr)
+
\bigl(\widetilde p^n-p_h^n\bigr).
\]
{
The estimator developed below concerns the second term, namely the spatial discretization error of the current time-discrete elliptic problem. The first term may contain the temporal consistency error, the error inherited from the previous numerical solution, and, on adaptive meshes, the effect of transferring data between consecutive finite element spaces. These contributions are not estimated separately in the present work.
}
\end{remark}

% : first, the upper bound—known as reliability—gives us a measure of how well the error is controlled, while the lower bound—referred to as efficiency—indicates the sharpness of the error estimate.
% {\color{red}YiYung, (Is this correct? Please modify as you wish)}

\subsection{Reliability}

{To provide the upper bound of the residual for the a posteriori error analysis, we first 
recall some approximation property for EG finite element space: for any $w\in H^2(\Tau_{h})$, there exists $w_h\in V^\text{EG}_{h,k}$ such that 
\begin{align}
    &\|w - w_h\|_{L^2(T)} \leq C_1 h^2_T \|w\|_{H^2(T)},\label{eq:inequality1} \\
    &\|w - w_h\|_{L^2(\gamma)} \leq C_2 h^{3/2}_T \|w\|_{H^2(T)},\label{eq:inequality2} \\
    &\|\nabla \left(w - w_h\right)\|_{L^2(\gamma)} \leq C_3 h^{1/2}_T \|w\|_{H^2(T)},
    \label{eq:inequality3}
\end{align}
for any $T\in\Tau_{h}$ and $\gamma\subset\partial T$, where the constants $C_1, C_2$, and $C_3$ depend on the regularity of mesh. Note that such $w_h\in V^\text{EG}_h$ is constructed mainly through a specific interpolation operator, details of which can be found in 
\cite{LEE_EG_Elliptic_Parabolic} (see equations (3.13) and (3.15)). Similar approximation properties for different finite element spaces can also be found in \cite{verfurth2013posteriori, AINSWORTH19971}, and these approximation properties are frequently used in both a priori and a posteriori error analysis. }

{
We note that such approximations $w_h$ can be constructed different ways, and it may belong to a finite element subspace of $V^\text{EG}_{h,k}$. For instance, selecting the conforming (continuous) part such that $w_h \in V^\text{CG}_{h,k} \subset V^\text{EG}_{h,k}$ is also a valid choice.  For example, the a posteriori error bounds for DG methods in \cite{georgoulis2011posteriori} were derived by choosing $w_h \in V^\text{CG}_{h,k}$ and then employing a suitable projection to proceed with the derivation of a posteriori bounds for an elliptic problem, followed by a parabolic problem.

Here, before we state Theorem \ref{thm:1}, we recall additional tools that quantify the relationship between jump or average values and trace values.
}
\begin{lemma}
For $v \in H^2(\Tau_{h})$, we have
\begin{align}
\sum_{\gamma\in\varepsilon_{h}}\int_\gamma \jump{v}^2 ds \leq 2 \sum_{T\in\Tau_{h}} \int_{\partial T} v^2 ds,\label{eq:jump_estimate}\\
    \sum_{\gamma\in\varepsilon_{h}}\int_\gamma \av{v}^2 ds \leq  \sum_{T\in\Tau_{h}} \int_{\partial T} v^2 ds,\label{eq:avg_estimate}
\end{align}
where $\partial T$ is the boundary of an element $T$.
\end{lemma}

\begin{proof}
See Lemma 2.27 in \cite{dolejvsi2015discontinuous}.
\end{proof}

{Then, we have the following theorem for the reliability.}
\begin{theorem}\label{thm:1} 
For $\theta \geq 0$, let $p^n_h$ and $\mathcal R^n$ be given by \eqref{eq:fullly_EG_discretization} and  \eqref{eq:residual}. Then,  the following upper bound for the residual holds:
\begin{equation}
\begin{split}
\|\mathcal R^n\|^2_* &\leq C_1 \sum_{T\in\Tau_{h,n}} (\eta^n_1)^2 
    + \hat{C}_2 \left(\sum_{\gamma\in\varepsilon_{h,n}^I} (\eta^n_2)^2 
        + \sum_{\gamma\in\varepsilon_{h,n}^N} (\eta^n_3)^2\right)
        + \theta \hat{C}_3\left(\sum_{\gamma\in\varepsilon^{I}_{h,n}} (\eta^n_4)^2
        + \sum_{\gamma\in \varepsilon_{h,n}^D} (\eta_5^n)^2\right)\\
        &\quad + \alpha \sqrt{2} \hat{C}_2 \left(\sum_{\gamma\in\varepsilon^I_{h,n}} (\eta^n_4)^2
        + \sum_{\gamma\in\varepsilon^D_{h,n}} (\eta_5^n)^2\right),
\end{split}
\end{equation}
where
\begin{align}
&\eta^n_1 = h^2_{T}\|f^n + \nabla \cdot (K\nabla p^n_h) - \dfrac{1}{\delta t}(p^n_h - \mathcal I_n p^{n-1}_h)\|_{L^2(T)}, \label{eq:eta1} \\
&\eta^n_2 = h^{3/2}_{\gamma} \left\|\jump{\mathbf n_\gamma\cdot (K\nabla p^n_h)}\right\|_{L^2(\gamma)}, \label{eq:eta2} \\
&\eta^n_3 = h^{3/2}_{\gamma} \|g^n_N+{\mathbf n_\gamma\cdot K \nabla p^n_h}
        \|_{L^2(\gamma)}, \label{eq:eta3}\\
&{\eta^n_4 = K_{\mathrm{max}(\gamma)} h^{1/2}_{\gamma} \|\jump{p^n_h}\|_{L^2(\gamma)}}, \label{eq:eta4}\\
& \eta^n_5 = K_{\mathrm{max}(\gamma)} h^{1/2}_{\gamma} \|g^n_D-p^n_h\|_{L^2(\gamma)}, 
\label{eq:eta5}
\end{align}
where $C_1$, $C_2$, and $C_3$ are from \eqref{eq:inequality1}-\eqref{eq:inequality3}, and $\hat{C}_2=C_2 \sigma^{3/2}$, and ${\hat{C}_3=C_3 d \sigma^{1/2}}$. 

\end{theorem}

\begin{proof}
The proof begins with \eqref{eq:residual} and refers to the expressions \eqref{eq:F_theta} - \eqref{eq:S_theta}. After rearranging the terms, we have the following.
\begin{equation}\label{upper1} 
\begin{split}
\mathcal R^n(w)
&= F^n_\theta(w) - S_\theta (p^n_h, w) \\
&= {\sum_{T\in\Tau_{h,n}} \left(f^n  -
        \dfrac{1}{\delta t} (p^n_h - \mathcal I_n p^{n-1}_h),w\right)_T - \sum_{T\in\Tau_{h,n}} \left(K\nabla p^n_h, \nabla w\right)_T + \sum_{\gamma \in \varepsilon^{ID}_{h,n}} \left(\av{\mathbf n_\gamma\cdot K\nabla p^n_h},\jump{w}\right)_\gamma }\\
&- \theta \sum_{\gamma \in \varepsilon^{ID}_{h,n}}\left(\jump{p^n_h},\av{\mathbf n_\gamma\cdot K\nabla w}
    \right)_\gamma
    - {\alpha \sum_{\gamma\in\varepsilon^{ID}_{h,n}} h_\gamma^{-1} K_{\mathrm{max}(\gamma)}\left(\jump{p^n_h},\jump{w}\right)_\gamma}     \\
& - \sum_{\gamma \in \varepsilon^N_{h,n}}(g^n_N, w)_\gamma
    +\theta \sum_{\gamma\in \varepsilon_{h,n}^D}(g^n_D,\mathbf n_\gamma\cdot K\nabla w)_\gamma + {\alpha \sum_{\gamma\in \varepsilon_{h,n}^D}h_\gamma^{-1} K_{\mathrm{max}(\gamma)}(g^n_D, w)_\gamma}.
    \end{split}
\end{equation}
Because of the following relation, 
\begin{equation*}
\begin{split}
    \sum_{T\in\Tau_{h,n}} \left(K\nabla p^n_h, \nabla w\right)_T &= {\sum_{T\in\Tau_{h,n}} \left(-\nabla \cdot K\nabla p^n_h, w\right)_T} + \sum_{T\in\Tau_{h,n}} \left(\mathbf n_T\cdot K\nabla p^n_h, w\right)_{\partial T}\\
    &= \sum_{T\in\Tau_{h,n}} \left(-\nabla \cdot K\nabla p^n_h, w\right)_T + \sum_{\gamma \in \varepsilon_{h,n}^I} \left[\left(\av{\mathbf n_\gamma\cdot K \nabla p^n_h}, \jump{w}\right)_\gamma +  \left(\jump{\mathbf n_\gamma\cdot K \nabla p^n_h}, \av{w}\right)_\gamma\right]\\
    & + \sum_{\gamma \in \varepsilon_{h,n}^D} \left({\mathbf n_\gamma\cdot K \nabla p^n_h}, {w}\right)_\gamma + \sum_{\gamma \in \varepsilon_{h,n}^N} \left({\mathbf n_\gamma\cdot K \nabla p^n_h}, {w}\right)_\gamma, 
\end{split}
\end{equation*}
{with} the Galerkin orthogonality \eqref{eq:orthogonality}, the above expression \eqref{upper1} becomes
{\begin{equation}\label{eq:identity}
    \begin{split}
        \mathcal R^n(w-w_h) &= F^n_\theta(w-w_h) - S_\theta (p^n_h, w-w_h) \\
        &= \sum_{T\in\Tau_{h,n}} \left(f^n + \nabla \cdot K\nabla p^n_h - \dfrac{1}{\delta t}(p^n_h - \mathcal I_n p^{n-1}_h),w-w_h \right)_T \quad  \\ 
        &\quad - \sum_{\gamma\in\varepsilon^I_{h,n}}\left(\jump{\mathbf n_\gamma\cdot K\nabla p^n_h}, \av{w-w_h}\right)_\gamma \quad\\
        &\quad - \theta \sum_{\gamma \in \varepsilon^{I}_{h,n}}\left(\jump{p^n_h},\av{\mathbf n_\gamma\cdot K\nabla (w-w_h)}
    \right)_\gamma \quad  \\
    &\quad +\theta \sum_{\gamma\in \varepsilon_{h,n}^D}(g^n_D-p^n_h,\mathbf n_\gamma\cdot K\nabla (w-w_h))_\gamma \quad 
    \\
    &{\quad -\alpha \sum_{\gamma\in\varepsilon^{I}_{h,n}} h_\gamma^{-1} K_{\mathrm{max}(\gamma)} \left(\jump{p^n_h},\jump{w-w_h}\right)_\gamma \quad} 
     \\
     &{\quad + \alpha \sum_{\gamma\in \varepsilon_{h,n}^D}h_\gamma^{-1} K_{\mathrm{max}(\gamma)} (g^n_D-p^n_h, w-w_h)_\gamma \quad}  \\
     &\quad - \sum_{\gamma \in \varepsilon^N_{h,n}}(g^n_N+{\mathbf n_\gamma\cdot K \nabla p^n_h}, w-w_h)_\gamma \quad  \\
     &:=\sum_{i=1}^7 \xi_i
    \end{split}
\end{equation}

In the following, we estimate each term $\xi_i$, $i=1,\dots,7$. 
First, by using the Cauchy-Schwarz inequality and \eqref{eq:inequality1}, $\xi_1$ has the upper bound as
\begin{equation}\label{eq:N1}
    \begin{split}
        |\xi_1| &\leq \sum_{T\in\Tau_{h,n}} \|f^n + \nabla \cdot K\nabla p^n_h - \dfrac{1}{\delta t}(p^n_h -\mathcal I_n p^{n-1}_h)\|_{L^2(T)}\|w - w_h\|_{L^2(T)}\\
        &\leq \sum_{T\in\Tau_{h,n}} \|f^n + \nabla \cdot K\nabla p^n_h - \dfrac{1}{\delta t}(p^n_h -\mathcal I_n p^{n-1}_h)\|_{L^2(T)} C_1 h^2_T \|w\|_{H^2(T)} \\
        &=C_1 \sum_{T\in\Tau_{h,n}} h^2_T\|f^n + \nabla \cdot K\nabla p^n_h - \dfrac{1}{\delta t}(p^n_h -\mathcal I_n p^{n-1}_h)\|_{L^2(T)}  \|w\|_{H^2(T)} \\
        & \leq  C_1 \left(\sum_{T\in\Tau_{h,n}}  h^4_T\|f^n + \nabla \cdot K\nabla p^n_h - \dfrac{1}{\delta t}(p^n_h - \mathcal I_n p^{n-1}_h)\|^2_{L^2(T)}  \right)^{1/2}
        \left(\sum_{T\in\Tau_{h,n}} \|w\|^2_{H^2(T)}\right)^{1/2}
    \end{split}
\end{equation}

For the rest $\xi_i$, $i=2,\dots,7$, by the Cauchy-Schwarz inequality, we estimate them as follows. 
\begin{equation}
    \begin{split}
        |\xi_2| 
        %& \leq \sum_{\gamma\in\varepsilon_{h,n}^I} \left\|\jump{\mathbf n_\gamma\cdot K\nabla p^n_h}\right\|_{L^2(\gamma)}\|\av{w - w_h}\|_{L^2(\gamma)}\\
        & \leq \left(\sum_{\gamma\in\varepsilon_{h,n}^I} \left\|\jump{\mathbf n_\gamma\cdot K\nabla p^n_h}\right\|_{L^2(\gamma)}^2\right)^{1/2}
    \left(\sum_{\gamma\in\varepsilon_{h,n}^I} \left\|\av{w - w_h}\right\|_{L^2(\gamma)}^2\right)^{1/2}, 
    \end{split}
\end{equation}

\begin{equation}
\begin{split}
    |\xi_3| 
    %&\leq \theta \sum_{\gamma\in\varepsilon_{h,n}^I} \|\jump{p^n_h}\|_{L^2(\gamma)} \| \av{\mathbf n_\gamma\cdot K\nabla \left(w - w_h\right) } \|_{L^2(\gamma)}\\
    &\leq \theta
    \left(\sum_{\gamma \in \varepsilon^{I}_{h,n}}
    \|\jump{p^n_h}\|_{L^2(\gamma)}^2
    \right)^{1/2}
    \left(\sum_{\gamma \in \varepsilon^{I}_h}
    \|\av{\mathbf n_\gamma\cdot K\nabla \left(w - w_h\right)}
    \|_{L^2(\gamma)}^2
    \right)^{1/2} ,
\end{split}
\end{equation}

\begin{equation}
\begin{split}
    |\xi_4| 
    %&\leq \theta \sum_{\gamma\in \varepsilon_{h,n}^D}
    %\|g^n_D-p^n_h\|_{L^2(\gamma)}
    %\|
    %\mathbf n_\gamma\cdot K\nabla\left(w - w_h\right)
    %\|_{L^2(\gamma)}\\
    &\leq \theta
    \left(\sum_{\gamma\in \varepsilon_{h,n}^D}
    \|g^n_D-p^n_h\|^2_{L^2(\gamma)}
    \right)^{1/2}
    \left(\sum_{\gamma\in \varepsilon_{h,n}^D}
    \|
    \mathbf n_\gamma\cdot K\nabla\left(w - w_h\right)
    \|_{L^2(\gamma)}^2
    \right)^{1/2}  , 
\end{split}
\end{equation}

\begin{equation}
    \begin{split}
        |\xi_5| 
        %& \leq
        %\alpha \sum_{\gamma\in\varepsilon^I_h} h_{\gamma}^{-1} K_{\mathrm{max}(\gamma)}\|\jump{p^n_h}\|_{L^2(\gamma)} \|\jump{w-w_h}\|_{L^2(\gamma)}\\
        & \leq
        \alpha \left(\sum_{\gamma\in\varepsilon^I_h} h_{\gamma}^{-2} K^2_{\mathrm{max}(\gamma)}\|\jump{p^n_h}\|^2_{L^2(\gamma)}\right)^{1/2} \left(\sum_{\gamma\in\varepsilon^I_h}\|\jump{w-w_h}\|^2_{L^2(\gamma)}\right)^{1/2}, 
    \end{split}
\end{equation}

\begin{equation}
    \begin{split}
        |\xi_6| 
        %& \leq
        %\alpha \sum_{\gamma\in\varepsilon^D_{h,n}} h_{\gamma}^{-1} K_{\mathrm{max}(\gamma)} \|g^n_D - p^n_h\|_{L^2(\gamma)} \|{w-w_h}\|_{L^2(\gamma)}\\
        & \leq
        \alpha \left(\sum_{\gamma\in\varepsilon^D_{h,n}} h_{\gamma}^{-2} K^2_{\mathrm{max}(\gamma)}\|g^n_D - p^n_h\|^2_{L^2(\gamma)}\right)^{1/2} \left(\sum_{\gamma\in\varepsilon^D_{h,n}}\|{w-w_h}\|^2_{L^2(\gamma)}\right)^{1/2} , 
    \end{split}
\end{equation}
and
\begin{equation}
    \begin{split}
        |\xi_7| 
        %&\leq  \sum_{\gamma\in\varepsilon_{h,n}^N}
        %\|g^n_N+{\mathbf n_\gamma\cdot K \nabla p^n_h}
        %\|_{L^2(\gamma)}
        %\|
        %w - w_h
        %\|_{L^2(\gamma)}\\
        &\leq \left(
        \sum_{\gamma\in\varepsilon_{h,n}^N}
        \|g^n_N+{\mathbf n_\gamma\cdot K \nabla p^n_h}
        \|_{L^2(\gamma)}^2
        \right)^{1/2} 
        \left(
        \sum_{\gamma\in\varepsilon_{h,n}^N}
        \|
        w - w_h
        \|_{L^2(\gamma)}^2
        \right)^{1/2} . 
    \end{split}
\end{equation}
% \phantom{a} \\ \indent 

The estimations will be done by collecting similar $|\xi_i|$, $i=1,\dots,7$. The first pair is $\xi_2$ and $\xi_7$. By the Cauchy-Schwarz inequality, \eqref{eq:avg_estimate}, \eqref{eq:inequality2}, and shape regularity $h_T / \rho_T \leq \sigma$, we obtain
\begin{equation}\label{eq:N2+N7}
    \begin{split}
        |\xi_2|+|\xi_7| &\leq {\left(\sum_{\gamma\in\varepsilon_{h,n}^I} \left\|\jump{\mathbf n_\gamma\cdot K\nabla p^n_h}\right\|_{L^2(\gamma)}^2\right)^{1/2}
        \left(\sum_{\gamma\in\varepsilon_{h,n}^I} \left\|\av{w - w_h}\right\|_{L^2(\gamma)}^2\right)^{1/2}} \\
        &\quad +{\left(
        \sum_{\gamma\in\varepsilon_{h,n}^N}
        \|g^n_N+{\mathbf n_\gamma\cdot K \nabla p^n_h}
        \|_{L^2(\gamma)}^2
        \right)^{1/2} 
        \left(
        \sum_{\gamma\in\varepsilon_{h,n}^N}
        \|
        w - w_h
        \|_{L^2(\gamma)}^2
        \right)^{1/2}} \\
        &\leq \left(
        \sum_{\gamma\in\varepsilon_{h,n}^I} \left\|\jump{\mathbf n_\gamma\cdot K\nabla p^n_h}\right\|_{L^2(\gamma)}^2+
        \sum_{\gamma\in\varepsilon_{h,n}^N}
        \|g^n_N+{\mathbf n_\gamma\cdot K \nabla p^n_h}
        \|_{L^2(\gamma)}^2
        \right)^{1/2} \left(\sum_{\gamma\in\varepsilon_{h,n}}
        \|
        \av{w - w_h}
        \|_{L^2(\gamma)}^2
        \right)^{1/2}\\
        &\leq \left(
        \sum_{\gamma\in\varepsilon_{h,n}^I} \left\|\jump{\mathbf n_\gamma\cdot K\nabla p^n_h}\right\|_{L^2(\gamma)}^2+
        \sum_{\gamma\in\varepsilon_{h,n}^N}
        \|g^n_N+{\mathbf n_\gamma\cdot K \nabla p^n_h}
        \|_{L^2(\gamma)}^2
        \right)^{1/2} \left(\sum_{T \in \Tau_{h,n}}
        \|
        w - w_h
        \|_{L^2(\partial T)}^2
        \right)^{1/2}\\
        &\leq  \left(
        \sum_{\gamma\in\varepsilon_{h,n}^I} \left\|\jump{\mathbf n_\gamma\cdot K\nabla p^n_h}\right\|_{L^2(\gamma)}^2+
        \sum_{\gamma\in\varepsilon_{h,n}^N}
        \|g^n_N+{\mathbf n_\gamma\cdot K \nabla p^n_h}
        \|_{L^2(\gamma)}^2
        \right)^{1/2}
        \left(\sum_{T\in\Tau_{h,n}} 
        C^2_2 h^{3}_T \|w\|_{H^2(T)}^2
        \right)^{1/2}\\
        &\leq \hat{C}_2 \left(
        \sum_{\gamma\in\varepsilon_{h,n}^I} h^{3}_{\gamma} \left\|\jump{\mathbf n_\gamma\cdot K\nabla p^n_h}\right\|_{L^2(\gamma)}^2+
        \sum_{\gamma\in\varepsilon_{h,n}^N} 
        h^{3}_{\gamma} \|g^n_N+{\mathbf n_\gamma\cdot K \nabla p^n_h}
        \|_{L^2(\gamma)}^2
        \right)^{1/2}
         \left(\sum_{T\in\Tau_{h,n}} 
        \|w\|_{H^2(T)}^2
        \right)^{1/2}, 
    \end{split}
\end{equation}
where $\hat{C}_2 = C_2 \sigma^{3/2}$. 
\phantom{a} 

For the second pair $\xi_3$ and $\xi_4$, we use the Cauchy-Schwarz inequality, \eqref{eq:avg_estimate}, \eqref{eq:inequality3}, and shape regularity to have  
\begin{equation}\label{eq:N3+N4}
    \begin{split}
        |\xi_3| + |\xi_4| &\leq 
\theta
    \left(\sum_{\gamma \in \varepsilon^{I}_{h,n}}
    \|\jump{p^n_h}\|_{L^2(\gamma)}^2
    \right)^{1/2}
    \left(\sum_{\gamma \in \varepsilon^{I}_{h,n}}
    \|\av{\mathbf n_\gamma\cdot K\nabla \left(w - w_h\right) }\|_{L^2(\gamma)}^2
    \right)^{1/2}\\
        &\quad +\theta
    \left(\sum_{\gamma\in \varepsilon_{h,n}^D}
    \|g^n_D-p^n_h\|^2_{L^2(\gamma)}
    \right)^{1/2}
    \left(\sum_{\gamma\in \varepsilon_{h,n}^D}
    \|\mathbf n_\gamma\cdot K\nabla\left(w - w_h\right)\|_{L^2(\gamma)}^2
    \right)^{1/2}\\
    &\leq
    \theta \left(
    \sum_{\gamma \in \varepsilon^{I}_{h,n}}
    \|\jump{p^n_h}\|_{L^2(\gamma)}^2 
    + 
    \sum_{\gamma\in \varepsilon_{h,n}^D}
    \|g^n_D-p^n_h\|^2_{L^2(\gamma)}\right)^{1/2} 
\left(
\sum_{\gamma \in \varepsilon^{ID}_{h,n}}
    \|\mathbf n_\gamma \cdot \av{K\nabla \left(w - w_h\right) }\|_{L^2(\gamma)}^2 
\right)^{1/2}\\
&\leq
    \theta \left(
    \sum_{\gamma \in \varepsilon^{I}_{h,n}}
    \|\jump{p^n_h}\|_{L^2(\gamma)}^2 
    + 
    \sum_{\gamma\in \varepsilon_{h,n}^D}
    \|g^n_D-p^n_h\|^2_{L^2(\gamma)}\right)^{1/2} 
\left(
\sum_{\gamma \in \varepsilon^{ID}_{h,n}}
    \|\av{K\nabla \left(w - w_h\right) }\|_{L^2(\gamma)}^2 
\right)^{1/2}\\
&\leq
    \theta \left(
    \sum_{\gamma \in \varepsilon^{I}_{h,n}}
    \|\jump{p^n_h}\|_{L^2(\gamma)}^2 
    + 
    \sum_{\gamma\in \varepsilon_{h,n}^D}
    \|g^n_D-p^n_h\|^2_{L^2(\gamma)}\right)^{1/2} 
\left(
\sum_{\gamma \in \varepsilon^{ID}_{h,n}}
    d^2 K^2_{\mathrm{max}(\gamma)}\|\av{\nabla \left(w - w_h\right) }\|_{L^2(\gamma)}^2 
\right)^{1/2}\\
&\leq \theta d
\left(
    \sum_{\gamma \in \varepsilon^{I}_{h,n}}
     K^2_{\mathrm{max}(\gamma)}\|\jump{p^n_h}\|_{L^2(\gamma)}^2 
    + 
    \sum_{\gamma\in \varepsilon_{h,n}^D}
    K^2_{\mathrm{max}(\gamma)} \|g^n_D-p^n_h\|^2_{L^2(\gamma)}\right)^{1/2} \left(\sum_{T\in\Tau_{h,n}} 
    \|\nabla (w - w_h)\|^2_{L^2(\partial T)}
    \right)^{1/2}\\
    &\leq \theta d
\left(
    \sum_{\gamma \in \varepsilon^{I}_{h,n}}
    K^2_{\mathrm{max}(\gamma)}\|\jump{p^n_h}\|_{L^2(\gamma)}^2 
    + 
    \sum_{\gamma\in \varepsilon_{h,n}^D}
    K^2_{\mathrm{max}(\gamma)}\|g^n_D-p^n_h\|^2_{L^2(\gamma)}\right)^{1/2} \left(\sum_{T\in\Tau_{h,n}} 
     C^2_3 h_T\|w\|^2_{H^2( T)}
    \right)^{1/2}\\
     &\leq \theta \hat{C}_3
\left( 
    \sum_{\gamma \in \varepsilon^{I}_{h,n}}
    K^2_{\mathrm{max}(\gamma)}  h_{\gamma} \|\jump{p^n_h}\|_{^2(\gamma)}^2 
    + 
    \sum_{\gamma\in \varepsilon_{h,n}^D}
    K^2_{\mathrm{max}(\gamma)}  h_{\gamma} \|g^n_D-p^n_h\|^2_{H^2(\gamma)}\right)^{1/2}\left(\sum_{T\in\Tau_{h,n}} 
    \|w\|^2_{H^2( T)}
    \right)^{1/2}, 
    \end{split}
\end{equation}
where $\hat{C}_3 = C_3 d \sigma^{1/2}$. 

By the Cauchy-Schwarz inequality, \eqref{eq:jump_estimate}, \eqref{eq:inequality2}, and shape regularity, the last pair $\xi_5$ and $\xi_6$ is estimated as 
\begin{equation}\label{eq:N5+N6}
    \begin{split}
        |\xi_5|+|\xi_6| &\leq \alpha \left(\sum_{\gamma\in\varepsilon^I_{h,n}} h_{\gamma}^{-2} K^2_{\mathrm{max}(\gamma)}\|\jump{p^n_h}\|^2_{L^2(\gamma)}\right)^{1/2} \left(\sum_{\gamma\in\varepsilon^I_{h,n}}\|\jump{w-w_h}\|^2_{L^2(\gamma)}\right)^{1/2} \\
        &+ 
        \alpha \left(\sum_{\gamma\in\varepsilon^I_{h,n}} h_{\gamma}^{-2} K^2_{\mathrm{max}(\gamma)}\|g^n_D - p^n_h\|^2_{L^2(\gamma)}\right)^{1/2} \left(\sum_{\gamma\in\varepsilon^I_{h,n}}\|{w-w_h}\|^2_{L^2(\gamma)}\right)^{1/2}\\
        \leq &  \alpha
        \left(\sum_{\gamma\in\varepsilon^I_{h,n}} h_{\gamma}^{-2} K^2_{\mathrm{max}(\gamma)} \|\jump{p^n_h}\|^2_{L^2(\gamma)}
        +
        \sum_{\gamma\in\varepsilon^D_{h,n}} h_{\gamma}^{-2} K^2_{\mathrm{max}(\gamma)} \|g^n_D - p^n_h\|^2_{L^2(\gamma)}
        \right)^{1/2} \left(\sum_{\gamma\in\varepsilon_{h,n}} \|\jump{w-w_h}\|^2_{L^2(\gamma)}\right)
        \\
        \leq & 
        \alpha  \left(\sum_{\gamma\in\varepsilon^I_{h,n}} h_{\gamma}^{-2} K^2_{\mathrm{max}(\gamma)} \|\jump{p^n_h}\|^2_{L^2(\gamma)}
        +
        \sum_{\gamma\in\varepsilon^D_{h,n}} h_{\gamma}^{-2} K^2_{\mathrm{max}(\gamma)} \|g^n_D - p^n_h\|^2_{L^2(\gamma)}
        \right)^{1/2}  \left(\sum_{T\in\Tau_{h,n}}2 \|w - w_h\|^2_{L^2(\partial T)}\right)^{1/2}\\
        \leq & 
        \alpha  \left(\sum_{\gamma\in\varepsilon^I_{h,n}} h_{\gamma}^{-2} K^2_{\mathrm{max}(\gamma)} \|\jump{p^n_h}\|^2_{L^2(\gamma)}
        +
        \sum_{\gamma\in\varepsilon^D_{h,n}} h_{\gamma}^{-2} K^2_{\mathrm{max}(\gamma)} \|g^n_D - p^n_h\|^2_{L^2(\gamma)}
        \right)^{1/2}  \left(\sum_{T\in\Tau_{h,n}}2 C_2^2h^3_T\|w\|^2_{H^2(T)}\right)^{1/2}\\
        \leq & 
        \alpha \sqrt{2} \hat{C}_2 \left(\sum_{\gamma\in\varepsilon^I_{h,n}} K^2_{\mathrm{max}(\gamma)} h_{\gamma} \|\jump{p^n_h}\|^2_{L^2(\gamma)}
        +
        \sum_{\gamma\in\varepsilon^D_{h,n}} K^2_{\mathrm{max}(\gamma)} h_{\gamma} \|g^n_D - p^n_h\|^2_{L^2(\gamma)}
        \right)^{1/2} \left(\sum_{T\in\Tau_{h,n}}\|w\|^2_{H^2(T)}\right)^{1/2}.
    \end{split}
\end{equation}

Finally, collecting \eqref{eq:N1} for $|\xi_1|$, \eqref{eq:N2+N7} for $|\xi_2|+|\xi_7|$, \eqref{eq:N3+N4} for $|\xi_3|+|\xi_4|$ , and
    \eqref{eq:N5+N6} for $|\xi_5|+|\xi_6|$,
we have 
\begin{equation*}
    \dfrac{\mathcal R^n(w)}{\|w\|_{H^2(\Tau_{h,n})}} =  \dfrac{\mathcal R^n(w-w_h)}{\|w\|_{H^2(\Tau_{h,n})}}\leq \sum_{i=1}^4 E_i,\qquad \forall w\in H^2(\Tau_{h,n}),
\end{equation*}
where 
\begin{align}
    E_1 &:= C_1  \left(\sum_{T\in\Tau_h} h^4_T \|f^n + \nabla \cdot (K\nabla p^n_h) - \dfrac{1}{\delta t}(p^n_h - \mathcal I_n p^{n-1}_h)\|^2_{L^2(T)}  \right)^{1/2},\\
    E_2 &:= \hat{C}_2\left(
        \sum_{\gamma\in\varepsilon_h^I}  h^{3}_\gamma \left\|\jump{\mathbf n_\gamma\cdot (K\nabla p^n_h)}\right\|_{L^2(\gamma)}^2+
        h^{3}_\gamma \sum_{\gamma\in\varepsilon_h^N} 
        \|g_N^n+{\mathbf n_\gamma\cdot K \nabla p^n_h}
        \|_{L^2(\gamma)}^2
        \right)^{1/2},\\
    E_3 &:= {\theta \hat{C}_3
\left(
    \sum_{\gamma \in \varepsilon^{I}_h}
    K^2_{\mathrm{max}(\gamma)} h_\gamma  \|\jump{p^n_h}\|_{L^2(\gamma)}^2 
    + 
    \sum_{\gamma\in \varepsilon_h^D} K^2_{\mathrm{max}(\gamma)} h_\gamma
    \|g_D^n-p^n_h\|^2_{L^2(\gamma)}\right)^{1/2},}\\
    E_4 &= {
        \alpha \sqrt{2} \hat{C}_2   \left(\sum_{\gamma\in\varepsilon^I_h}  K^2_{\mathrm{max}(\gamma)} h_{\gamma} \|\jump{p^n_h}\|^2_{L^2(\gamma)}
        +
        \sum_{\gamma\in\varepsilon^D_h}K^2_{\mathrm{max}(\gamma)} h_{\gamma}  \|g_D^n - p^n_h\|^2_{L^2(\gamma)}
        \right)^{1/2}.}
\end{align}
}
\end{proof}

\subsection{Efficiency}
{
Next, we establish the {efficiency}, which provides a local lower bound for the residual. These bounds demonstrate that if the residual is small, the error estimator remains controlled, ensuring that mesh refinement is guided effectively and preventing excessive refinement of the local mesh.  

To derive such a local lower bound, we utilize element- or edge-bubble functions, denoted by
{
\(\chi_T \text{ and } \chi_\gamma \), respectively, } 
as described in \cite{verfurth2013posteriori} and \cite{AINSWORTH19971}. These functions are polynomials that continuously vanish to zero on the boundaries of local elements,
{ and \( 0 \leq \chi_T, \chi_\gamma \leq 1 \).}
%With these properties, we can analyze the residual locally.  
We note that for element-bubble functions, \( \chi_T(\mathbf{x}) = 0 \) if \( \mathbf{x} \not\in T \). On the other hand, edge-bubble functions \( \chi_{\gamma} \) take nonzero values on elements that share \( \gamma \), which we denote by \( \tilde{\gamma} = T_- \cup T_+ \). In particular, we define the bubble functions for a given interior, Dirichlet, or Neumann edge as \( \chi_{\gamma,I} \), \( \chi_{\gamma,D} \), and \( \chi_{\gamma,N} \), respectively.  

}

{Here, we recall the following lemma provides some useful estimates for the bubble functions.}
\begin{lemma}\label{lemma_bubble}
The bubble functions $\chi_T$ and $\chi_\gamma$ the following estimates
\begin{align}
\|v\|_{L^2(T)} \leq C_{L,1} \|\chi_T^{1/2} v\|_{L^2(T)}, \label{eq:bb_lem_1} \\
\|v\|_{L^2(\gamma)} \leq C_{L,2} \|\chi_\gamma^{1/2} v\|_{L^2(\gamma)}, \label{eq:bb_lem_2} \\
\|\chi_\gamma v\|_{L^2(\tilde{\gamma})} \leq C_{L,3} h^{1/2}_\gamma \|v\|_{L^2(\gamma)}, \label{eq:bb_lem_3} \\
\|\chi_\gamma v\|_{H^1(\tilde{\gamma})} \leq C_{L,4} h^{-1/2}_\gamma \|v\|_{L^2(\gamma)}, \label{eq:bb_lem_4} \\
\|\chi_\gamma v\|_{H^2(\tilde{\gamma})} \leq C_{L,5} h^{-3/2}_\gamma \|v\|_{L^2(\gamma)}, \label{eq:bb_lem_5} 
\end{align}
where $v$ is a polynomial. The constants $C_{L,1}, \cdots, C_{L,5}$ depend only on the mesh and shape regularity. 
\begin{proof}
See Proposition 1.4 in \cite{verfurth2013posteriori} for \eqref{eq:bb_lem_1}-\eqref{eq:bb_lem_4}. The inequality \eqref{eq:bb_lem_5} can be proven in the same way for \eqref{eq:bb_lem_4}.
\end{proof}
\end{lemma}
Next, the following
inverse inequality~\cite{riviere2008discontinuous} 
for any polynomial function $v$  defined on an element $T$,  will also be used 
\begin{equation}\label{eq:inverse}
\|v\|_{H^2(T)} \leq C_I h^{-2}_T \|v\|_{L^2(T)}. 
\end{equation}

%\textcolor{red}{
%To derive the lower bounds about \eqref{eq:eta1}-\eqref{eq:eta5}, the inequalities in Lemma \ref{lemma_bubble} will be used, so the terms in the norms of \eqref{eq:eta1}-\eqref{eq:eta5} have to be polynomial, which is typically obtained by $L^2$-projection. Here, two projections are introduced. Firstly, for the term $u$ in the norms of \eqref{eq:eta1}, \eqref{eq:eta3}, and \eqref{eq:eta5}, we perform an element-wise projection, i.e.,
%\begin{equation}
%u_h \in V^{EG}_{h,k} \phantom{a} (u-u_h,v_h)_T=0 \phantom{a} \forall v_h \in M^k_0(\Tau_h).
%\end{equation}
%Secondly, for the term $[u](=u_{T_+}-u_{T_-})$ in the norms of \eqref{eq:eta2}, and \eqref{eq:eta4}, we perform an patch-wise projection, i.e.,
%\begin{equation}
%\begin{split}
%&u_{h,T_{+}} \in M^k_0(\Tau_h) \phantom{a} (u_{T_{+}} - u_{h,T_{+}},v_h)_{T_{+}}=0 \phantom{a} \forall v_h \in M^k_0(\Tau_h), \\
%&u_{h,T_{-}} \in M^k_0(\Tau_h) \phantom{a} (u_{T_{-}}-u_{h,T_{-}},v_h)_{T_{-}}=0 %\end{split}
%\end{equation}
%Then, along $\gamma(=T_+ \cap T_{-})$, $u_{h,T_{+}} - u_{h,T_{-}} = [u_h] \in M^k_0(\Tau_h)$.
%}

With these mathematical tools, we present the following theorem for efficiency. 
\begin{theorem}\label{thm:efficiency}
For $\theta \geq 0$, let $p^n_h$ and $\mathcal R^n$ be given by \eqref{eq:fullly_EG_discretization} and  \eqref{eq:residual}. Then,  the following local efficiency estimates hold:

\begin{subequations}
    \begin{alignat}{2}
        &\eta^n_1 = h^2_T\|A\|_{L^2(T)} \leq C_{L,1}^{2} \left(h^2_T\|\overline{A}-A\|_{L^2(T)} + C_I \|\mathcal R\|_{*(T)}\right), \\
       & \eta^n_2 = h^{3/2}_\gamma\|B\|_{L^2(\gamma)} \leq C^2_{L,2}( h^{3/2}_\gamma\|\overline{B}-B\|_{L^2(\gamma)} + C_{L,5} \|\mathcal R\|_{*(\tilde{\gamma})}\notag \\
        &\quad \quad \quad \quad \quad \quad \ + C_{L,3} h^{2}_\gamma \|A\|_{L^2(\tilde{\gamma})} + \theta C_{L,4} K h_\gamma \|D\|_{L^2(\gamma)}),\\
        &\eta^n_3 = h^{3/2}_\gamma \|C\|_{L^2(\gamma)} \leq C^2_{L,2} \left( h^{3/2}_\gamma\|\overline{C}-C\|_{L^2(\gamma)} + C_{L,5} \|\mathcal R\|_{*(\tilde{\gamma)}} + C_{L,3} h^{2}_\gamma \|A\|_{L^2(\tilde{\gamma})} \right),\\
        %& \eta^n_4 = K_{\mathrm{max}(\gamma)} h^{1/2}_\gamma \|D\|_{L^2(\gamma)} \leq C^2_{L,2} \bigg( K_{\mathrm{max}(\gamma)} h^{1/2}_\gamma\|\overline{D}-D\|_{L^2(\gamma)} + K_{\mathrm{max}(\gamma)} h^{1/2}_\gamma\|D\|_{L^2(\gamma)} + C_{L,5} \|\mathcal{R}\|_{*(\tilde{\gamma})}  \notag \\
%&\quad \quad \quad \quad \quad \quad  \ + C_{I,3} h^{2}_{\gamma} \|A\|_{L^2(\tilde{\gamma})}  + C_{I,3} h^{2}_{\gamma}\|B\|_{L^2(\gamma)} + \theta C_{I,4} h_\gamma K_{\mathrm{max}(\gamma)} \|D\|_{L^2(\gamma)} \bigg)\\
        %& \alpha \eta^n_5 = \alpha K_{\mathrm{max}(\gamma)} h^{1/2}_\gamma  \|E\|_{L^2(\gamma)} \leq C^2_{L,2} \bigg(\alpha K_{\mathrm{max}(\gamma)} h^{1/2}_\gamma  \|\overline{E}-E\|_{L^2(\gamma)} + C_{L,5}  \|\mathcal R\|_{*(\tilde{\gamma)}}\notag \\
&\quad \quad \quad \quad \quad \quad  \ +  C_{L,3} h^{2}_\gamma \|A\|_{L^2(T)} + \theta C_{L,4}  K_{\mathrm{max}(\gamma)} h_\gamma \|E\|_{L^2(\gamma)} \bigg),\label{eq:E_final}\\
&\eta^n_4 = K_{\mathrm{max}(\gamma)} h^{1/2}_\gamma \|D\|_{L^2(\gamma)} \leq C^2_{L,2} \bigg(  2C_D h^{1/2}_\gamma \|\overline{D}-D\|_{L^2(\gamma)} +  2C_D h^{1/2}_\gamma \|\overline{D}\|_{L^2(\gamma)} + C_{L,5} \|\mathcal{R}\|_{*(\tilde\gamma)}  \notag \\ 
& \quad \quad \quad \quad \quad \quad + C_{L,3} h^{2}_{\gamma} \|A\|_{L^2(\tilde{\gamma})}  + C_{L,3} h^{2}_{\gamma}\|B\|_{L^2(\gamma)}  \bigg),\\ 
&\alpha \eta^n_5=\alpha  K_{\mathrm{max}(\gamma)} h^{1/2}_\gamma \|E\|_{L^2(\gamma)} \leq C^2_{L,2} \bigg( C_E h^{3/2}_\gamma\|\overline{E}-E\|_{L^2(\gamma)} + C_{L,5} \|\mathcal R\|_{*(\tilde{\gamma})} \notag \\ 
&\quad \quad \quad \quad \quad \quad+  C_{L,3} h^{2}_\gamma \|A\|_{L^2(T)} +  C_E h^{3/2}_\gamma\|\overline{E}\|_{L^2(\gamma)} \bigg), 
    \end{alignat}
\end{subequations}
where
\begin{subequations}
    \begin{alignat}{2}
        A &= f^n + \nabla \cdot (K\nabla p^n_h) - \dfrac{1}{\delta t}(p^n_h - \mathcal I_n p^{n-1}_h), \label{eq:eta1} \\
        B &= \jump{\mathbf n_\gamma\cdot (K\nabla p^n_h)}, \label{eq:eta2} \\
        C &= g^n_N+{\mathbf n_\gamma\cdot K \nabla p^n_h}, \label{eq:eta3}\\
        D &= \jump{p^n_h}, \label{eq:eta4}\\
        E &= g^n_D-p^n_h, \label{eq:eta5}
    \end{alignat}
\end{subequations}
and $\overline{A}$, $\overline{B}$, $\overline{C}$, $\overline{D}$, $\overline{E}$ are the local perturbations of $A$, $B$, $C$, $D$, $E$ in $M^k_0(\Tau_{h,n})$, defined through  $L^2$--projection; see Remark \ref{remark:efficiency}. The constants $C_{L,1},\ldots,C_{L,5}$ are those appearing in Lemma~\ref{lemma_bubble}, and $C_D=K_{\mathrm{max}  (\gamma)} \mathrm{max}(1, \theta C_{L,4} h^{1/2}_\gamma)$, $C_E  = K_{\mathrm{max}(\gamma)}  \mathrm{max}(\alpha h^{-1}_\gamma, \theta C_{L,4} h^{-1/2}_\gamma )$.

\end{theorem}

\begin{proof}
% Let $\overline{A}$, $\overline{B}$, $\overline{C}$, $\overline{D}$, $\overline{E}$ be $L^2$ projection of $A$, $B$, $C$, $D$, $E$ on $M^k_0(\Tau_{h,n})$. We estimate the lower bounds in the order $A$ to $E$.  \\

\noindent (\romannumeral 1) Estimation for $A=f^n + \nabla \cdot (K\nabla p^n_h) - \dfrac{1}{\delta t}(p^n_h - \mathcal I_n p^{n-1}_h)$

By replacing  $w-w_h$ by $\overline{A} \chi_T$ in (\ref{eq:identity}), we have
\begin{equation}\label{eq:A1}
\mathcal R(\overline{A} \chi_T) = \int_T A \overline{A} \chi_T dx, 
\end{equation}
which satisfies
\begin{equation}\label{eq:A}
\int_T \overline{A}^2 \chi_T dx = \int_T \chi_T \overline{A}(\overline{A}-A) dx + \mathcal R(\overline{A} \chi_T). 
\end{equation}
Using the Cauchy-Schwarz inequality, the first term on the right-hand side of (\ref{eq:A}) is estimated as
\begin{equation}\label{eq:first_A}
\int_T \chi_T \overline{A}(\overline{A}-A) dx \leq \|\chi_T \overline{A}\|_{L^2(T)} \|\overline{A}-A\|_{L^2(T)} \leq  \| \overline{A}\|_{L^2(T)} \|\overline{A}-A\|_{L^2(T)},
\end{equation}
and by applying (\ref{eq:dual_norm}) and (\ref{eq:inverse}), the second term is estimated as  
\begin{equation}\label{eq:second_A}
\mathcal R(\overline{A} \chi_T) \leq \|\mathcal R\|_{*(T)} \|\overline{A} \chi_T\|_{H^2(T)} \leq \|\mathcal R\|_{*(T)} C_I h^{-2}_T\|\overline{A} \chi_T\|_{L^2(T)} \leq \|\mathcal R\|_{*(T)} C_I h^{-2}_T\|\overline{A}\|_{L^2(T)}.
\end{equation}
Collecting (\ref{eq:first_A}) and (\ref{eq:second_A}), (\ref{eq:A}) becomes
\begin{equation}\label{eq:sum_A}
\int_T \overline{A}^2 \chi_T dx  \leq \| \overline{A}\|_{L^2(T)} \|\overline{A}-A\|_{L^2(T)} + C_I h^{-2}_T\|\mathcal R\|_{*(T)} \|\overline{A}\|_{L^2(T)},
\end{equation}
and then by (\ref{eq:bb_lem_1})
\begin{equation}
\|\overline{A}\|_{L^2(T)} \leq C_{L,1}^{2} \left(\|\overline{A}-A\|_{L^2(T)} + C_I h^{-2}_T\|\mathcal R\|_{*(T)}\right).
\end{equation}
By the triangle inequality, we can have
\begin{equation}\label{eq:A_final}
\|A\|_{L^2(T)} \leq C_{L,1}^{2} \left(\|\overline{A}-A\|_{L^2(T)} + C_I h^{-2}_T\|\mathcal R\|_{*(T)}\right) . 
\end{equation}

\phantom{a} \\
\noindent (\romannumeral 2) Estimation for $B=\jump{\mathbf n_\gamma\cdot (K\nabla p^n_h)}$

Again, by substituting $\overline{B} \chi_{\gamma,I}$ for $w-w_h$ in (\ref{eq:identity}), then we have
\begin{equation}\label{eq:S_interior_edge}
    \begin{split}
        \mathcal R(\overline{B} \chi_{\gamma,I}) &= \sum_{T\in\Tau_h} \left(f^n + \nabla \cdot (K\nabla p^n_h) - \dfrac{1}{\delta t}(p^n_h - \mathcal I_n p^{n-1}_h), \overline{B} \chi_{\gamma,I} \right)_T - \sum_{{\gamma'}\in\varepsilon^I_h}\left(\jump{\mathbf n_{\gamma'}\cdot (K\nabla p^n_h)}, \av{\overline{B} \chi_{\gamma,I} }\right)_{\gamma'}\\
        &- \theta \sum_{{\gamma'} \in \varepsilon^{I}_h}\left(\jump{p^n_h},\av{\mathbf n_{\gamma'}\cdot (K\nabla (\overline{B} \chi_{\gamma,I}) )}
    \right)_{\gamma'} +\theta \sum_{{\gamma'}\in \varepsilon_h^D}\left(g^n_D-p^n_h, \mathbf n_{\gamma'}\cdot K\nabla (\overline{B} \chi_{\gamma,I}) \right)_{\gamma'}
    \\
    & -\alpha \sum_{{\gamma'}\in\varepsilon^{I}_h} h_{\gamma'}^{-1} K_{\mathrm{max}(\gamma)} \left(\jump{p^n_h},\jump{\overline{B} \chi_{\gamma,I}}\right)_{\gamma'}
     + \alpha \sum_{{\gamma'}\in \varepsilon_h^D}h_{\gamma'}^{-1} K_{\mathrm{max}(\gamma)} \left(g^n_D-p^n_h, \overline{B} \chi_{\gamma,I} \right)_{\gamma'}\\
     & - \sum_{{\gamma'} \in \varepsilon^N_h}\left(g^n_N+{\mathbf n_{\gamma'}\cdot K \nabla p^n_h}, \overline{B} \chi_{\gamma,I} \right)_{\gamma'} . 
    \end{split}
\end{equation}

Since $\chi_{\gamma,I}$ is continuous, $\jump{\overline{B} \chi_{\gamma,I}}=0$ and $\{\overline{B} \chi_{\gamma,I}\}=\overline{B} \chi_{\gamma,I}$ along edge $\gamma$.
Also, the terms for Dirichlet and Neumann boundaries disappear because $\chi_{\gamma,I}$ is zero on the boundaries. 
Thus, (\ref{eq:S_interior_edge}) is reduced to
\begin{equation}\label{eq:S_interior_reduced}
\mathcal R(\overline{B} \chi_{\gamma,I})=\left(A, \overline{B} \chi_{\gamma,I}\right)_{\tilde{\gamma}} - \left(B, \overline{B} \chi_{\gamma,I}\right)_{\gamma} - \theta \left(D, \{\mathbf n_\gamma\cdot K\nabla (\overline{B} \chi_{\gamma,I}) \}\right)_\gamma ,
\end{equation}
which satisfies
\begin{equation}\label{eq:B}
\int_\gamma \overline{B}^2 \chi_{\gamma,I} ds = \int_\gamma \chi_{\gamma,I} \overline{B}(\overline{B}-B) ds - \mathcal R(\overline{B}\chi_{\gamma,I}) + (A, \overline{B}\chi_{\gamma,I})_{\tilde{\gamma}} - \theta \left(D, \{\mathbf n_\gamma\cdot K\nabla (\overline{B} \chi_{\gamma,I})\}\right)_\gamma .
\end{equation}

Each term on the right-hand side is estimated as follows. By the Cauchy-Schwarz inequality, 
\begin{equation}\label{eq:B1}
\int_\gamma \chi_{\gamma,I} \overline{B}(\overline{B}-B) ds \leq \|\chi_{\gamma,I} \overline{B}\|_{L^2(\gamma)} \|\overline{B}-B\|_{L^2(\gamma)} \leq \|\overline{B}\|_{L^2(\gamma)} \|\overline{B}-B\|_{L^2(\gamma)}.
\end{equation}
By the Cauchy-Schwarz inequality, \eqref{eq:dual_norm}, and \eqref{eq:bb_lem_5}, 
\begin{equation}\label{eq:B2}
- \mathcal R(\overline{B}\chi_{\gamma,I}) \leq \|\mathcal R\|_{*(\tilde{\gamma})} \|\overline{B}\chi_{\gamma,I}\|_{H^2(\tilde{\gamma})} \leq \|\mathcal R\|_{*(\tilde{\gamma})} C_{L,5} h^{-3/2}_\gamma\|\overline{B}\|_{L^2(\gamma)}.
\end{equation}
Again, by the Cauchy-Schwarz inequality and \eqref{eq:bb_lem_3}, 
\begin{equation}\label{eq:B3}
(A, \overline{B}\chi_{\gamma,I})_{\tilde{\gamma}} \leq \|A\|_{L^2(\tilde{\gamma})} ||\overline{B}\chi_{\gamma,I}\|_{L^2(\tilde{\gamma})} \leq \|A\|_{L^2(\tilde{\gamma})} C_{L,3} h^{1/2}_\gamma||\overline{B}\|_{L^2(\gamma)}.
\end{equation}
By the Cauchy-Schwarz inequality and \eqref{eq:bb_lem_4}, the last term can be estimated as
\begin{equation}\label{eq:B4}
\begin{split}
-\left(D, \{\mathbf n_\gamma\cdot K\nabla (\overline{B} \chi_{\gamma,I} ) \}\right)_\gamma &\leq \|D\|_{L^2(\gamma)}  \|\{\mathbf n_\gamma\cdot K\nabla (\overline{B} \chi_{\gamma,I})\}\|_{L^2(\gamma)}  \\
& \leq \|D\|_{L^2(\gamma)} \|K\nabla (\overline{B} \chi_{\gamma,I}) \|_{L^2(\tilde{\gamma})} \\
&\leq \|D\|_{L^2(\gamma)} C_{L,4} h^{-1/2}_\gamma K \|\overline{B} \|_{L^2(\gamma)} .
\end{split}
\end{equation}

Collecting \eqref{eq:B1}-\eqref{eq:B4} and using \eqref{eq:bb_lem_2}, \eqref{eq:B} becomes
\begin{equation}
\|\overline{B}\|_{L^2(\gamma)} \leq C^2_{L,2}\left( \|\overline{B}-B\|_{L^2(\gamma)} + C_{L,5} h^{-3/2}_\gamma \|\mathcal R\|_{*(\tilde{\gamma})} + C_{L,3} h^{1/2}_\gamma \|A\|_{L^2(\tilde{\gamma})} + \theta C_{L,4} K h^{-1/2}_\gamma \|D\|_{L^2(\gamma)}\right) ,
\end{equation}
and by the triangle inequality,
\begin{equation}\label{eq:B_final}
\|B\|_{L^2(\gamma)} \leq C^2_{L,2}\left( \|\overline{B}-B\|_{L^2(\gamma)} + C_{L,5} h^{-3/2}_\gamma \|\mathcal R\|_{*(\tilde{\gamma})} + C_{L,3} h^{1/2}_\gamma \|A\|_{L^2(\tilde{\gamma})} + \theta C_{L,4} K h^{-1/2}_\gamma \|D\|_{L^2(\gamma)}\right) .
\end{equation}

\phantom{a} \\
\noindent (\romannumeral 3) Estimation for $C=g_N+{\mathbf n_\gamma\cdot K \nabla p^n_h}$

By substitute $\overline{C} \chi_{\gamma,N}$ for $w-w_h$ in \eqref{eq:identity}, then we have
\begin{equation}
\mathcal R(\overline{C} \chi_{\gamma,N}) = (A, \overline{C} \chi_{\gamma,N})_{\tilde{\gamma}} - (C, \overline{C} \chi_{\gamma,N})_{\gamma},
\end{equation}
thus
\begin{equation}
\int_{\gamma} \overline{C}^2 \chi_{\gamma,N} ds = \int_{\gamma} \chi_{\gamma,N}\overline{C}(\overline{C}-C) ds - \mathcal R(\overline{C} \chi_{\gamma,N}) + (A, \overline{C} \chi_{\gamma,N})_{\tilde{\gamma}}.
\end{equation}
Similar to the estimation for $B$, by the Cauchy-Schwarz inequality, \eqref{eq:bb_lem_3}, \eqref{eq:bb_lem_5}, and the triangle inequality, then we obtain
\begin{equation}
\|\overline{C}\|_{L^2(\gamma)} \leq C^2_{L,2}\left( \|\overline{C}-C\|_{L^2(\gamma)} + C_{L,5} h^{-3/2}_\gamma \|\mathcal R\|_{*(\tilde{\gamma})} + C_{L,3} h^{1/2}_\gamma \|A\|_{L^2(\tilde{\gamma})} \right)
\end{equation}
and
\begin{equation}\label{eq:C_final}
\|C\|_{L^2(\gamma)} \leq C^2_{L,2} \left( \|\overline{C}-C\|_{L^2(\gamma)} + C_{L,5} h^{-3/2}_\gamma \|\mathcal R\|_{*(\tilde{\gamma})} + C_{L,3} h^{1/2}_\gamma \|A\|_{L^2(\tilde{\gamma})} \right).
\end{equation}

\phantom{a} \\
\noindent (\romannumeral 4) Estimation for $D=\jump{p^n_h}$

With $h_\gamma \overline{D}\chi_{\gamma,I}$ instead of $w-w_h$ in \eqref{eq:identity}, we have
\begin{equation}
\mathcal R(h_\gamma \overline{D}\chi_{\gamma,I}) = (A, h_\gamma \overline{D}\chi_{\gamma,I})_{\tilde{\gamma}} - (B, h_\gamma \overline{D}\chi_{\gamma,I})_\gamma - \theta(D, \mathbf n_\gamma \cdot K \nabla(h_\gamma \overline{D}\chi_{\gamma,I}))_\gamma.
\end{equation}

\begin{equation}
\begin{split}
K_{\mathrm{max}(\gamma)} \int_{\gamma}  \overline{D}^2 \chi_{\gamma,I} ds &= K_{\mathrm{max}(\gamma)} \int_{\gamma} \chi_{\gamma,I}\overline{D}(\overline{D}-D) ds + K_{\mathrm{max}(\gamma)} \int_{\gamma}  \chi_{\gamma,I} \overline{D}D ds + \mathcal R(h_\gamma \overline{D}\chi_{\gamma,I}) \\
&- (A, h_\gamma \overline{D}\chi_{\gamma,I})_{\tilde{\gamma}} + (B, h_\gamma \overline{D}\chi_{\gamma,I})_\gamma + \theta(D, \mathbf n_\gamma \cdot K \nabla(h_\gamma \overline{D}\chi_{\gamma,I}))_\gamma.
\end{split}
\end{equation}

Now, we estimate each term in similar ways as shown for $B$ and $C$. 
\begin{equation}
K_{\mathrm{max}(\gamma)} \int_{\gamma}  \chi_{\gamma,I}\overline{D}(\overline{D}-D) ds \leq K_{\mathrm{max}(\gamma)} \|\overline{D}\|_{L^2(\gamma)}\|\overline{D}-D\|_{L^2(\gamma)}
\end{equation}

\begin{equation}
K_{\mathrm{max}(\gamma)} \int_{\gamma} \chi_{\gamma,I} \overline{D}D ds \leq K_{\mathrm{max}(\gamma)} \|\overline{D}\|_{L^2(\gamma)}\|D\|_{L^2(\gamma)}
\end{equation}

\begin{equation}
\begin{split}
\mathcal R(h_\gamma \overline{D} \chi_{\gamma,I}) \leq  \|\mathcal{R}\|_{*(\tilde{\gamma})} \|h_\gamma \overline{D} \chi_{\gamma,I}\|_{H^2(\tilde{\gamma})} \leq C_{L,5} h^{-1/2}_\gamma \|\mathcal{R}\|_{*(\tilde{\gamma})} \|\overline{D}\|_{L^2(\gamma)}.
\end{split}
\end{equation}

\begin{equation}
\begin{split}
-(A, h_\gamma \overline{D}\chi_{\gamma,I})_{\tilde{\gamma}}  \leq \|A\|_{L^2(\tilde{\gamma})} \|h_\gamma\overline{D}\chi_{\gamma,I}\|_{L^2(\tilde{\gamma})} 
\leq  C_{L,3} h^{3/2}_{\gamma} \|A\|_{L^2(\tilde{\gamma})} \|\overline{D}\|_{L^2(\gamma)}.
\end{split}
\end{equation}

\begin{equation}
(B, h_\gamma \overline{D} \chi_{\gamma,I})_{\gamma} \leq  \|B\|_{L^2(\gamma)} \|h_\gamma\overline{D}\chi_{\gamma,I}\|_{L^2(\gamma)} \leq C_{L,3} h^{3/2}_\gamma \|B\|_{L^2(\gamma)} \|\overline{D}\|_{L^2(\gamma)}.
\end{equation}

\begin{equation}
(D, \mathbf n_\gamma \cdot (K\nabla(h_\gamma \overline{D} \chi_{\gamma,I})) )_\gamma  \leq \|D\|_{L^2(\gamma)}  K_{\mathrm{max}(\gamma)} h_\gamma\|\nabla(\overline{D}\chi_{\gamma,I})\|_{L^2(\gamma)} \leq  C_{L,4} h^{1/2}_\gamma K_{\mathrm{max}(\gamma)} \|D\|_{L^2(\gamma)} \| \overline{D}\|_{L^2(\gamma)}.
\end{equation}

Collecting all the above estimates, we have
\begin{equation}\label{eq:D_final}
\begin{split}
K_{\mathrm{max}(\gamma)} \|\overline{D}\|_{L^2(\gamma)} &\leq C^2_{L,2} \bigg( K_{\mathrm{max}(\gamma)} \|\overline{D}-D\|_{L^2(\gamma)} + K_{\mathrm{max}(\gamma)} \|D\|_{L^2(\gamma)} + C_{L,5} h^{-1/2}_\gamma\|\mathcal{R}\|_{*(\tilde\gamma)}   \\
& + C_{L,3} h^{3/2}_{\gamma} \|A\|_{L^2(\tilde{\gamma})}  + C_{L,3} h^{3/2}_{\gamma}\|B\|_{L^2(\gamma)} + \theta C_{L,4} h^{1/2}_\gamma K_{\mathrm{max}(\gamma)} \|D\|_{L^2(\gamma)} \bigg).
\end{split}
\end{equation}
{
The triangle inequality $\|D\| \leq \|\overline{D}-D\| + \|\overline{D}\|$ leads to 
\begin{equation}
\begin{split}
K_{\mathrm{max}(\gamma)} \|D\|_{L^2(\gamma)} &\leq C^2_{L,2} \bigg(  2C_D\|\overline{D}-D\|_{L^2(\gamma)} +  2C_D\|\overline{D}\|_{L^2(\gamma)} + C_{L,5} h^{-1/2}_\gamma\|\mathcal{R}\|_{*(\tilde\gamma)}   \\
& + C_{L,3} h^{3/2}_{\gamma} \|A\|_{L^2(\tilde{\gamma})}  + C_{L,3} h^{3/2}_{\gamma}\|B\|_{L^2(\gamma)}  \bigg), 
\end{split}
\end{equation}
where $C_D := K_{\mathrm{max} (\gamma)} \times  \mathrm{max}(1, \theta C_{L,4} h^{1/2}_\gamma)$
}

\phantom{a} \\
\noindent (\romannumeral 5) Estimation for $E= g_D-p^n_h  $

With $\overline{E}\chi_{\gamma,D}$, \eqref{eq:identity} becomes
\begin{equation}
\mathcal R(\overline{E} \chi_{\gamma,D}) = (A, \overline{E} \chi_{\gamma,D})_{\tilde{\gamma}} + \theta (E,\mathbf n_\gamma\cdot K\nabla (\overline{E} \chi_{\gamma,D}) )_\gamma + \alpha h_\gamma^{-1} K_{\mathrm{max}(\gamma)} (E, \overline{E} \chi_{\gamma,D} )_\gamma, 
\end{equation}
which fulfills
\begin{equation}
\begin{split}
\alpha h^{-1}_\gamma K_{\mathrm{max}(\gamma)} \int_\gamma \overline{E}^2 \chi_{\gamma,D} ds &= \alpha h^{-1}_\gamma K_{\mathrm{max}(\gamma)} \int_\gamma \chi_{\gamma,D} \overline{E}(\overline{E}-E) ds + \mathcal R(\overline{E}\chi_{\gamma,D}) \\
& - (A, \overline{E}\chi_{\gamma,D})_{\tilde{\gamma}} - \theta \left(E, \mathbf n_\gamma\cdot (K\nabla (\overline{E} \chi_{\gamma,D}) ) \right)_\gamma .
\end{split}
\end{equation}
By a similar estimation to those shown above, we have
\begin{equation}\label{eq:E_final}
\begin{split}
\alpha h^{-1}_\gamma K_{\mathrm{max}(\gamma)} \|\overline{E}\|_{L^2(\gamma)} \leq C^2_{L,2} &\bigg(\alpha h^{-1}_\gamma K_{\mathrm{max}(\gamma)} \|\overline{E}-E\|_{L^2(\gamma)} + C_{L,5} h^{-3/2}_\gamma \|\mathcal R\|_{*(\tilde{\gamma)}} \\
&+  C_{L,3} h^{1/2}_\gamma \|A\|_{L^2(T)} + \theta C_{L,4} h^{-1/2}_\gamma K_{\mathrm{max}(\gamma)} \|E\|_{L^2(\gamma)} \bigg),
\end{split}
\end{equation}
{
and the following is obtained by triangle inequality
\begin{equation}
\begin{split}
\alpha h^{-1}_\gamma K_{\mathrm{max}(\gamma)} \|E\|_{L^2(\gamma)} \leq C^2_{L,2} &\bigg( C_E\|\overline{E}-E\|_{L^2(\gamma)} + C_{L,5} h^{-3/2}_\gamma \|\mathcal R\|_{*(\tilde{\gamma})} \\
&+  C_{L,3} h^{1/2}_\gamma \|A\|_{L^2(T)} +  C_E\|\overline{E}\|_{L^2(\gamma)} \bigg)  
\end{split}
\end{equation}
where $C_E = K_{\mathrm{max}(\gamma)} \times \mathrm{max}(\alpha h^{-1}_\gamma, \theta C_{L,4} h^{-1/2}_\gamma )$.
}

Finally, we finish the proof by multiplying \eqref{eq:A_final},\eqref{eq:B_final}, \eqref{eq:C_final}, \eqref{eq:D_final},\eqref{eq:E_final} by $h^2_T$, $h^{3/2}_\gamma$, $h^{3/2}_\gamma$, $h^{1/2}_\gamma$, and $h^{3/2}_\gamma$, respectively. 

\end{proof}

\begin{remark}\label{remark:efficiency}
The local perturbations of \( A, B, C, D, E \) in \eqref{eq:eta1}–\eqref{eq:eta5} are polynomial approximations of themselves. Therefore, we may apply Lemma \ref{lemma_bubble} to prove Theorem \ref{thm:efficiency}. These approximations can be defined directly through the \( L^2 \)-projection, e.g., \( \overline{A} \), or in an analogous manner for \( \overline{B}, \dots, \overline{E} \). For further details, we refer readers to \cite{verfurth2013posteriori, AINSWORTH19971, book_a_posteriori_error}.

\end{remark}
\clearpage
\newpage

\clearpage
\newpage
\section{Numerical example}
\label{sec:num}

In this section, we present the numerical experiments of the EG approximation of equation \eqref{eqn:main}. All numerical examples are solved in a two-dimensional domain ($d = 2$) using IIPG~($\theta = 0$) and NIPG~($\theta =1$) formulation with quadrilateral ($\mathbb Q_k$) elements. The implementation is carried out using the  deal.II finite element library \cite{dealII94}. 

We recall that the time discretization is carried out with uniform time step $\delta t = T/N$, and 
the fully-discrete problem \eqref{eq:fullly_EG_discretization} at each time step $t = t_n$, $n \in \{1,..., N\}$. 
The numerical error over the simulation time interval  $[0,t_n]\subset[0,T]$ is measured in the following norm
\begin{equation}
			% \|p_h - p \|_{l^\infty(0,t_n;L^2)}:= \max_{0\leq m \leq n} \|p^m_h - p^m\|_{L^2(\Omega)}, \text{ and } 
			\|p_h - p \|_{l^\infty(0,t_n;H^1)}:= \max_{1\leq m \leq n} \vertiii{p^m_h - p^m}_{H^1(\Tau_{h,m})}.
\end{equation}

The error estimators at a given computational time $t_n$ derived from Theorem \ref{thm:1} will be defined as
\begin{equation}\label{eq:estimators_total}
	\eta^{n} := \sqrt{\sum_{T\in\Tau_{h,n}}(\eta_T^{n})^2},
\end{equation}
where $\eta_T^n$ for $T\in\Tau_{h,n}$ is local error estimators, {which} is given by
\begin{equation}\label{eq:estimators_local}
	\eta_T^{n} := \sqrt{(\eta^n_1)^2 + 0.5 \times \sum_{\gamma \subset \partial T,\gamma\in \varepsilon^I_{h,n}} \left(\alpha (\eta^n_4)^2 + (\eta_2^n)^2 \right) + \sum_{\gamma\subset\partial T,\gamma \in \varepsilon^N_{h,n}} (\eta^n_3)^2 + \sum_{\gamma\subset\partial T, \gamma\in\varepsilon^D_{h,n}} \alpha (\eta_5^n)^2}.
\end{equation} 
The factor $0.5$ avoids double counting of interior edge contributions, since each interior edge $\gamma\in\varepsilon^I_{h,n}$ is shared by two neighboring elements. The parameter $\alpha>0$ is the penalty parameter from the EG discretization. In the analysis, penalty-dependent constants are absorbed into the generic constants in Theorems~\ref{thm:1} and~\ref{thm:efficiency}; in the numerical local estimator, the corresponding edge and boundary indicators are multiplied by $\alpha$ to reflect the penalty scaling used in the scheme.

To analyze the convergence behavior of the total error estimations, we define
\begin{equation}
	\eta_{l^\infty(0,t_n)} := \max_{1\leq m \leq n} \sum_{T\in\Tau_{h,m}}\eta^m_T.
\end{equation}
The effectivity indices are then given by
\begin{equation} 
% \text{EI}_{n,L^2} := \dfrac{\eta_{l^\infty(0,t_n)}}{|p_h-p|_{l^\infty(0,t_n;L^2)}},\quad \text{and} \quad 
\text{EI}_{n} := \dfrac{\eta_{l^\infty(0,t_n)}}{\|p_h-p\|_{l^\infty(0,t_n;H^1)}}, \end{equation} which are used to assess whether $\eta^n$ is a good estimation of the error. If these indices become unbounded, the error estimation overestimates the error, indicating that the computed error bounds are too large relative to the true error. Conversely, if these indices decrease to zero, the error estimation underestimates the error, implying that the estimator fails to adequately capture the actual error magnitude. {Here we note that the $H_1$-norm is used here to assess the performance of the error estimator instead of residual norm \ref{eq:residual}; however, the numerical example in the following section will show the error estimator is also reliable and efficient with respect to reduce the approximation error in $H_1$-norm.
}

Finally, we define the minimal mesh size over the simulation time $[0,T]$ by 
\begin{equation}
    h_\text{min} = \min_{0\leq n \leq N} h_\text{min}^n,
\end{equation}
where $h_\text{min}^n$ is the smallest mesh size in the triangulation $\Tau_{h,n}$ at time step $t_n$.

\subsection{Adaptive algorithm}\label{sec:algorithm}
The computations incorporate the error estimators \eqref{eq:estimators_total}-\eqref{eq:estimators_local}, and the $h$--adaptive refinement at each time step \( t_n \) consists of two main parts:
\begin{itemize}
    \item {Coarsening:} Mark and coarsen \( \theta_\text{coarse}\% \) of the elements \( T \) in the triangulation \( \mathcal{T}_{h,n-1} \) with the lowest local error estimators.
    
    \item {Refinement and Iterative Solving:} Mark and refine the elements in \( \mathcal{M} \subset \Tau_{h,n} \) using the D\"orfler marking strategy \cite{doi:10.1137/0733054}, where \( \theta_\text{refine}\% \) is chosen such that  
    \begin{equation}
        \theta_\text{refine} \eta^n \leq \sum_{T\in\mathcal{M}} \eta^n_T.
    \end{equation}
    The problem \eqref{eqn:main} is then solved on the updated triangulation. This process is repeated iteratively until the prescribed tolerance \( \tau > 0 \) is satisfied:
    $  \eta^{n} < \tau.$
\end{itemize}
We outline the adaptive algorithm  in the following:
\begin{algorithm}[H]
	\caption{$h$--Adaptive algorithm at $t_n$: updating ($p^{n-1}_h,\Tau_{h,n-1})$ to ($p^{n}_h,\Tau_{h,n}$)}\label{alg:h-adapt}
	\begin{algorithmic}%[1]\label{algorithm}
		\Procedure{$(p^n_h,\Tau_{h,n})=$Adaptive}{$\tau,\theta_\text{coarse},\theta_\text{refine};p^{n-1}_h,\Tau_{h,n-1}$}
		\State Compute $\eta^{n-1}_T$ on $\Tau_{h,n-1}$ with $p^{n-1}_h$

		\State $\Tau_{h,n-1}^\text{coarse} =\text{Coarse}(\Tau_{h,n-1})$: $\Tau_{h,n-1}^\text{coarse}$ is obtained by coarsening $T\in\Tau_{h,n-1}$ such that 
		
		\qquad \qquad \qquad \qquad \qquad \qquad $\eta^{n-1}_T \leq \theta_\text{coarse}\max_{T'\in\Tau_{h,n-1}}\eta^{n-1}_{T'}$
%		\State (Project $p^{n-1}_h$ to $\Tau_{h,n-1}^\text{coarse}$)
		\State Set $i = 0$, and set $\Tau_{h}^{(i)} = \Tau_{h,n-1}^\text{coarse}$
		
		 \quad Solve (1) on $\Tau_{h}^{(i)}$ to have $p^{(i)}_h$, and compute $\eta^{(i)}$
		\While{$\eta^{(i)} < \tau$}
		
		\State $\Tau_h^{\text{Refine}}=\text{Refine}(\Tau_{h}^{(i)})$ by D\"ofler marking strategy:
		
		\quad \quad \qquad Find the smallest subset $\mathcal M\subset \Tau_h^{(i)}$ such that
		
		\qquad \qquad \qquad \qquad \qquad \qquad $\theta_\text{refine}\eta \leq \sum_{T\in\mathcal M}\eta_T$
		
		\quad \quad \qquad $\Tau_h^{\text{Refine}}$ is obtained by refining all the meshes in $\mathcal M$
%		\State (Project $p^{n-1}_h$ to $\Tau_{h,n-1}^\text{Refine}$)
		\State Set $i = i+1$, and set $\Tau_{h}^{(i)} = \Tau_{h,n-1}^\text{refine}$
		
		\quad\ \quad Solve (1) on $\Tau_{h}^{(i)}$ to have $p^{(i)}_h$, and compute $\eta^{(i)}$
		\EndWhile
		\State \textbf{return} $(p^n_h,\Tau_{h,n}) = (p^{(i)}_h,\Tau_{h}^{(i)})$
	\EndProcedure
\end{algorithmic}
\end{algorithm}
The parameters $\theta_\text{coarse}$ and $\theta_\text{refine}$ control the coarsening and refinement steps in the adaptive procedure. The parameter $\theta_\text{refine}$ determines the fraction of the total error estimator used in the D"orfler marking strategy. Larger values generally mark more elements for refinement, leading to faster error reduction at the expense of additional degrees of freedom. The parameter $\theta_\text{coarse}$ controls how aggressively elements with small local error indicators are selected for coarsening. In the numerical examples, these parameters are chosen empirically to balance accuracy and computational cost.

On adaptive meshes, hanging nodes may appear at nonmatching interfaces between refined and unrefined elements; see Figure~\ref{fig:egq1-hanging-node} for an illustration in the EG-$\mathbb Q_1$ case. For the continuous $\mathbb Q_k$ component, the degrees of freedom associated with hanging nodes are constrained by interpolation from the neighboring coarse-grid edge or face.

To illustrate this in the $\mathbb Q_1$ case, let $x_1$ and $x_2$ be the endpoints of a coarse edge and let $x_m$ be the midpoint hanging node, as shown in Figure~\ref{fig:egq1-hanging-node}. The hanging-node value is constrained by interpolation from the neighboring coarse-edge degrees of freedom and is therefore not treated as an independent unknown. This preserves the global continuity of the continuous $\mathbb Q_1$ component across nonmatching interfaces. For higher-order $\mathbb Q_k$ elements, analogous Lagrange interpolation constraints are imposed using the degrees of freedom on the corresponding coarse-grid edge or face. The piecewise constant enrichment is discontinuous by construction and therefore does not require hanging-node constraints.

\begin{figure}[htbp]
\centering
\begin{tikzpicture}[scale=2.2]

%------------------------------------------------
% Coordinates
%------------------------------------------------
% left coarse element
\coordinate (A) at (0,0);
\coordinate (B) at (0,1);
\coordinate (C) at (1,0);
\coordinate (D) at (1,1);

% right refined 2 by 2 grid
\coordinate (E) at (2,0);
\coordinate (F) at (2,0.5);
\coordinate (G) at (2,1);

\coordinate (M) at (1,0.5);
\coordinate (R) at (1.5,0);
\coordinate (S) at (1.5,0.5);
\coordinate (T) at (1.5,1);

%------------------------------------------------
% Mesh
%------------------------------------------------
% left coarse element
\draw[thick] (A) rectangle (D);

% right four refined elements
\draw[thick] (C) rectangle (G);
\draw[thick] (R) -- (T);
\draw[thick] (M) -- (F);

% nonmatching interface
\draw[very thick,dashed] (C) -- (D);

%------------------------------------------------
% CG Q1 degrees of freedom: filled circles
%------------------------------------------------
\foreach \P in {A,B,C,D,E,F,G,M,R,S,T}
{
    \fill (\P) circle (1.5pt);
}

% labels on the nonmatching interface
\node[below=3pt] at (C) { $x_1$};
\node[left,red] at (M) {$x_m$};
\fill[red] (M) circle (1.1pt) ;
\node[above=3pt] at (D) { $x_2$};

%------------------------------------------------
% DGQ0 degrees of freedom: open circles
%------------------------------------------------
\node[draw, fill=white, minimum size=5pt, inner sep=0pt] at (0.5,0.5) {};

\foreach \x/\y/\name in {
    1.25/0.25/U_1,
    1.75/0.25/U_2,
    1.25/0.75/U_3,
    1.75/0.75/U_4}
{
    \node[draw, fill=white, minimum size=5pt, inner sep=0pt] at (\x,\y) {};
    \node[above=3pt] at (\x,\y) {};
}

%------------------------------------------------
% Titles
%------------------------------------------------
% \node[align=center] at (0.5,1.18) {\small coarse element};
% \node[align=center] at (1.5,1.18) {\small four refined elements};

% interface label

%------------------------------------------------
% Legend on the right-hand side
%------------------------------------------------
\begin{scope}[shift={(2.35,0.-0.1)}]
    \fill (0,1) circle (1.5pt);
    \node[anchor=west] at (0.08,1) {\small nodal CG-$\mathbb Q_1$ degree of freedom};

    \fill[black] (0, 0.78) circle (1.5pt);
    \fill[red] (0, 0.78) circle (1.1pt);
    \node[anchor=west] at (0.08,0.78) {\small nonmatching interface with hanging node $x_m$};

    \node[draw, fill=white, minimum size=5pt, inner sep=0pt] at (0,0.56) {};
    \node[anchor=west] at (0.08,0.56) {\small element-wise DG-$\mathbb Q_0$ degree of freedom};

\node[anchor=west] at (0.08,0.25)
{\small $p_h(x_m)=\tfrac12 p_h(x_1)+\tfrac12 p_h(x_2)$};

\end{scope}

\end{tikzpicture}
\caption{
Illustration of the EG-$\mathbb Q_1$ degrees of freedom on a nonmatching grid. The filled circles denote the continuous $\mathbb Q_1$ degrees of freedom, while the open circles denote the element-wise constant enrichment. The hanging-node value at $x_m$ is constrained by interpolation from the coarse-edge values at $x_1$ and $x_2$.
}
\label{fig:egq1-hanging-node}
\end{figure}

% To analyze the convergence behavior of the total error estimations  from $t_0$ to $t_m$, the following quantity is used
% \begin{equation}
% 	\eta_{l^\infty(0,t_m)} := \max_{0\leq n \leq m} \sum_{T\in\Tau_{h,n}}\eta^n_\text{total}.
% \end{equation}
% Furthermore, the following efficiency indices are introduced:
% \begin{equation}
% 		\text{EI}_{m,L^2} := \dfrac{\eta_{l^\infty(0,t_m)}}{\|p_h-p\|_{l^\infty(0,t_m;L^2)}},\quad \text{and }\quad 
% 		\text{EI}_{m,H^1} := \dfrac{\eta_{l^\infty(0,t_m)}}{\|p_h-p\|_{l^\infty(0,t_m;H^1)}}.
% \end{equation}
\subsection{Example 1. A solution with a singularity in a L-shaped domain}\label{eqxample1}

In the L--shaped 
domain $(x,y)\in\Omega = (-1,1)^2\backslash[0,1]^2$  within time interval $t\in[0,T],$ where $T=0.5$, 
we consider the following exact solution which is written in polar coordinate $(r,\phi)$ as 
\begin{equation}
    p(x,y,t) = \hat{p}(r,\phi,t) = \sin(\dfrac{\pi}{2} t) r^{2/3}\sin(\dfrac{2}{3} \phi).\label{eq:p_1}
\end{equation}
Here, $\phi$ is defined as clockwise angle against the $x$--axis on the $\Omega$. 
The Dirichlet boundary condition imposed on $\partial \Omega$, the initial condition, and the source term $f$ are chosen to match $p$  with $K \equiv I_2$ in (\ref{eqn:main}), where $I_2$ is the identity matrix in $\mathbb R^{2\times2}$.  We note that the given $p$ exhibits a singularity at the re-entrant corner located at the origin of $\Omega$. In addition, $\nabla \cdot (\nabla p) = 0$ in this case. In the next section, we consider a case where $\nabla \cdot (\nabla p) \neq 0$, and we discuss the differences in the performance of the error estimator.

\subsubsection{Example 1.1. Convergence test with the uniform mesh}

First, we test the convergence using uniformly refined meshes, and compute the total error estimator $\eta^n$. Then, we will compare the error and $\eta^n$ 
to the mesh-adaptive case to verify the effectiveness of our adaptive algorithm.
Here, we test with six uniformly refined meshes where each cycle has 
{$h_0 = 2^{-1}$ to $h_0 = 2^{-6}$}.
We use EG-$\mathbb Q_1$ with the penalty term set to $\alpha = 1$  and backward Euler time discretization with a fixed time step $\delta t = 0.01$.

We compute the error by employing $\|\cdot \|_{l^\infty(0,T;H^1)}$ norm, and 
the error convergence rate with respect to the total numbers of degrees of freedom (Dofs) at the last time step is computed using the formula: 
\begin{equation*}
    \text{Order}_\text{Dofs}(j) := d\times \frac{\log(\text{Error}(j)) - \log(\text{Error}(j-1))}{\log(\text{Dofs}(j-1)) - \log(\text{Dofs}(j))},
\end{equation*}
for each cycle $j$.
The convergence results are shown in Figure  \ref{fig:Example1.1_error_a}.
where the optimal convergence rate of 0.66 in $\|\cdot \|_{l^\infty(0,T;H^1)}$ norm is observed.
%In Figure \ref{fig:Example1.1_LastErrorEstimation}, we show the error estimation $\eta^N$, $N = 50$ at $t = 0.5$, and the same convergence rate with the error is observed. 
In Figure \ref{fig:Example1.1_estimation_over_time}, for the given fourth-cycle~(where $h_0= 2^{-4}$), we show the $\eta^n$ \eqref{eq:estimators_total} for each time step  $n = 0,\cdots,N$. We  observe that $\eta^n$ values increase over time in this uniformly refined case.  In addition, the error estimator $\eta ^N$ at the final time step $N$ for each cycle is shown in Figure \ref{fig:Example1.1_LastErrorEstimation}.

% {\color{red} In Figure X, 
% for the given cycle X (where $h=X$), we show the $\eta^n$ \eqref{eq:estimators_total} for each time step $n=0,\cdots, N$.
%  We observe that $\eta^n$ value.... 
% Next, the error estimator $\eta^N$, at the final time step $N$ is shown for each cycles in Figure X. } 

%the error estimator $\eta^N$, which is $\eta^n$ \eqref{eq:estimators_total} at the final time step $N$ is shown for each cycles in Figure X. } 

% {\color{red}In Figure 1 b), we show that  $\eta_{l^\infty(0,T)}$ converges in the second order, where
% $\eta_{l^\infty(0,T)}:={\max_{0\leq n\leq N}}\eta^n$.}

% \begin{figure}[!h]
% \centering
% \begin{subfigure}[b]{0.45\textwidth}
% \centering
% \includegraphics[width=\textwidth]{Figure//Error.jpg}
% \caption{Error $\|p - p_h\|_{l^\infty(0,T;H^1)}$}
% \label{fig:Example1.1_error_a}
% \end{subfigure}
% %
% \begin{subfigure}[b]{0.45\textwidth}
% \centering
% \includegraphics[width=\textwidth]{Figure//Untitled-design-4.png}
% \caption{ Error estimation $\eta^n$ for selected cycles }
% \label{fig:Example1.1_estimation_over_time}
% \end{subfigure}
% \caption{Example 1.1-1.2. Convergence tests with uniform and adaptive mesh. }
% \label{fig:Example1.1_error}
% \end{figure}

\begin{figure}[H]
    \centering
\begin{subfigure}[b]{0.45\textwidth}
\centering
\makebox[\textwidth][c]{%
  \hspace*{-1.2cm}%
  \includegraphics[width=\textwidth]{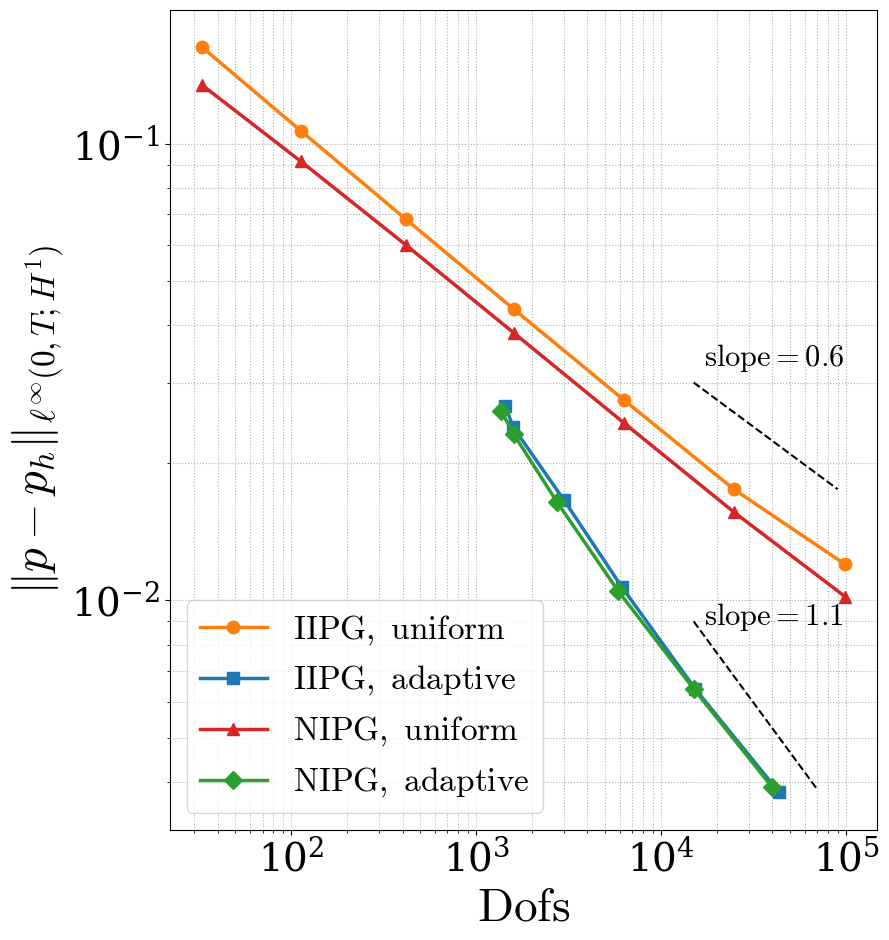}%
}
\caption{Error $\|p - p_h\|_{l^\infty(0,T;H^1)}$}
\label{fig:Example1.1_error_a}
\end{subfigure}\quad 
\begin{subfigure}[b]{0.45\textwidth}
\centering
\makebox[\textwidth][c]{%
  \hspace*{-0.7cm}%
\includegraphics[width=\textwidth]{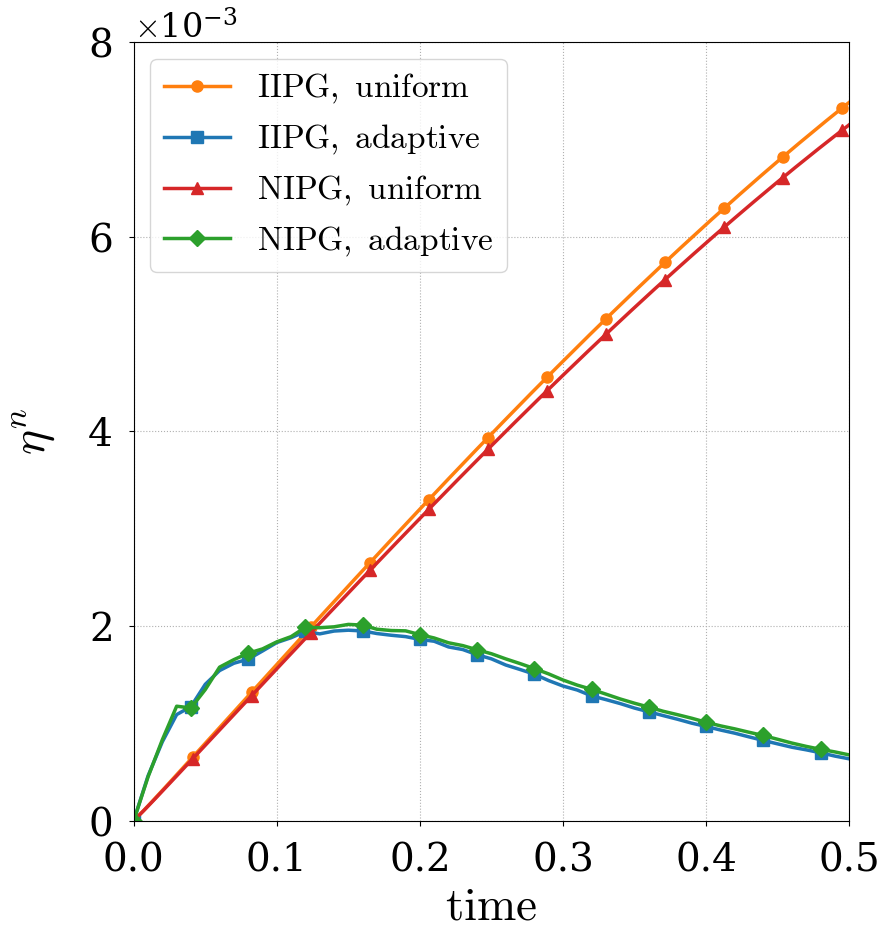}
}
\caption{Error estimation $\eta^n$ for selected cycles}
\label{fig:Example1.1_estimation_over_time}
\end{subfigure}
\caption{Example 1.1--1.2. Convergence tests with uniform and adaptive meshes using the IIPG and NIPG formulations. }
\label{fig:Example1.1_error}
\end{figure}

%\begin{figure}[H]
 %   \centering
    % \begin{subfigure}[b]{0.5\textwidth}
  %      \centering
        %\includegraphics[width=0.5\textwidth]{Figure/Error.jpg}
        % 
        
    % \end{subfigure}%
    %\hfill
    % \begin{subfigure}[b]{0.5\textwidth}
    %     \centering
    %     \includegraphics[width=\textwidth]{Figure/Last_Estimation.jpg}
    %     \caption{Error estimation $\eta^N$ at $t = 0.5.$}
    % \label{fig:Example1.1_LastErrorEstimation}
    % \end{subfigure}%
    % \label{fig:Example1.1_error}
%\end{figure}

% \begin{table}[H]
%     \centering
%     \begin{tabular}{c|c|c|c|c|c}
%   $h$& Dofs& $\|\cdot\|_{l^\infty(0,T;H^1)}$ &   Order & $\eta_{l^\infty(0,T)}  $& Order \\
%    \hline
% $2^{-2}$& 113  &0.1067 &   - &0.2045 &    - \\
% $2^{-3}$& 417  &0.0683 &0.64 &0.0511 & 2.00\\
% $2^{-4}$& 1601 &0.0434 &0.65 &0.0128 & 2.00 \\
% $2^{-5}$& 6273 &0.0275 &0.66 &0.0032 & 2.00  \\
%     \end{tabular}
%     \caption{Example 1.1. A convergence results with uniform mesh (EG-$\mathbb Q_1$). }
%     \label{tab:1}
% \end{table}

\subsubsection{Example 1.2. Convergence test with the adaptive mesh}

Next, we test convergence using mesh adaptivity. We perform \( h \)-refinement over six cycles, adaptively refining the mesh at each time step with Algorithm \ref{alg:h-adapt} using a refinement threshold of \( \theta_{\text{refine}} = 10\% \). However, we do not coarsen the mesh or impose any tolerance criterion.
For each cycle, the initial mesh is uniform with mesh sizes given by $h_0 = 1$ to $h_0=2^{-5}$. 
The minimal size for the mesh after the refinement for each cycle is $h_\text{min} = 2^{-7},   2^{-7},   2^{-8},   2^{-9},  2^{-12}, \text{and } 2^{-14}$, respectively.  
The time step is fixed at $\delta t = 0.01$.
The adaptive refinement patterns obtained with the IIPG and NIPG formulations are visually indistinguishable. Therefore, we show only the IIPG mesh in Figure~\ref{fig:Example1-2_mesh}. For the fifth cycle ($h_0=2^{-4}$), the adaptive meshes at $t=0.1$, $0.25$, and $0.5$ are illustrated in Figure~\ref{fig:Example1-2_mesh}.

\begin{figure}[H]
    \centering
    %\hspace*{-1.5cm}
    \begin{subfigure}[b]{0.33\textwidth}
        \centering
        \includegraphics[width=\textwidth]{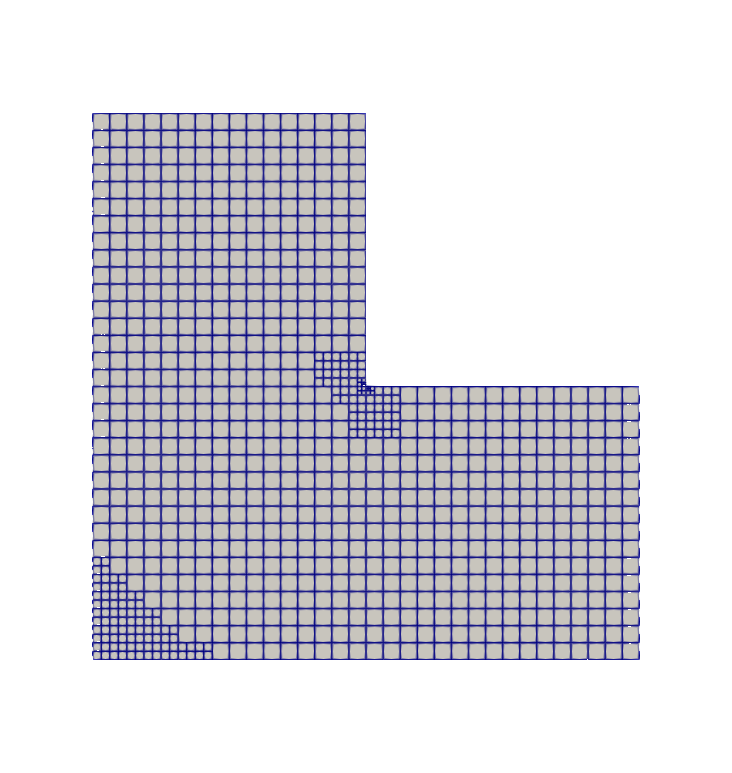}
        \caption{$t = 0.1$}
        \label{fig:p1_mesh1}
    \end{subfigure}%
    %\hfill
    \begin{subfigure}[b]{0.33\textwidth}
        \centering
        \includegraphics[width=\textwidth]{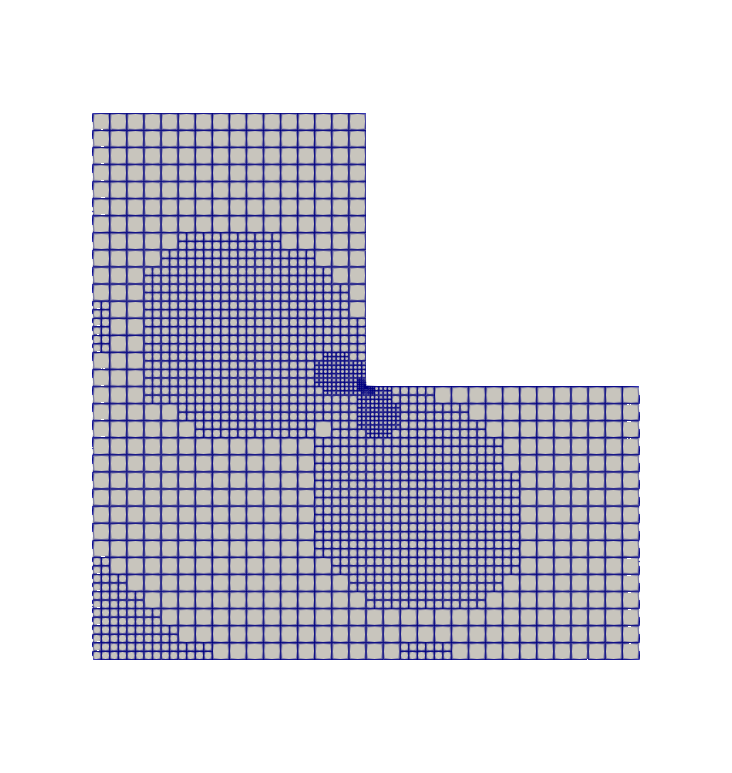}
        \caption{$t=0.25$}
    \end{subfigure}%
    %\hfill
    \begin{subfigure}[b]{0.33\textwidth}
        \centering
        \includegraphics[width=\textwidth]{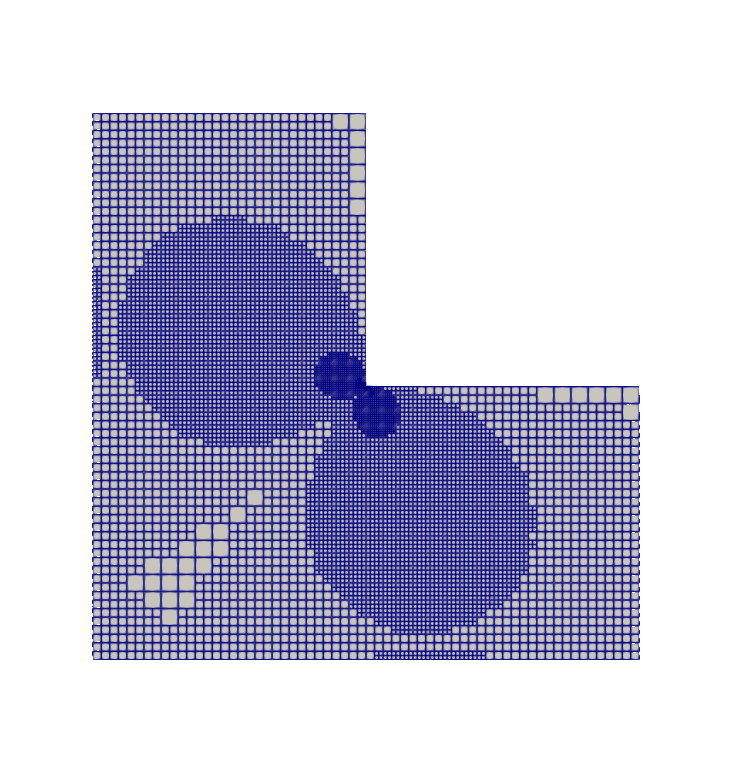}
        \caption{$t=0.5$}
        \label{fig:p1_mesh3}
    \end{subfigure}%
    \caption{Example 1.2. The $h$--adaptive meshes at $t=0.1$, $0.25$, and $0.5$ in the fifth cycle ($h_0=2^{-4}$).}
    \label{fig:Example1-2_mesh}
\end{figure}

The convergence of the errors is illustrated in Figure \ref{fig:Example1.1_error_a}, and the convergence rate achieve the optimal order with 1.1 in $\|\cdot\|_{l^\infty(0,T;H^1)}$ norm. {We note the error with the adaptivity is less than the uniform mesh with the similar number of Dofs. The comparison is shown in Figure\ref{fig:Example1.1_error_a}. 
Moreover, we emphasize that we obtain the  better convergence rate with the adaptivity.

% {\color{red} different notation between the initial mesh and refine mesh $h_0$ $h_{\text{min}}$}

Next, in Figure \ref{fig:Example1.1_estimation_over_time}, we show the $\eta^n$ for each time step $n = 0,\ldots, N$, 
for the cycle $i=3$, where the initial mesh is $h_0=2^{-3}$, and the minimal mesh size is $h_\text{min} = 2^{-8}$.
We observe that $\eta^n$ initially increases when the mesh is too coarse, but it soon decreases as the mesh is refined adaptively. We emphasize that at the final time step, the number of Dofs is similar to that of the uniformly refined mesh with $h_0=2^{-4}$, yet the error estimator $\eta^N$ is significantly lower when using the adaptive mesh. Figure~\ref{fig:Example1.1_LastErrorEstimation} shows $\eta^N$ for each cycle, illustrating that the error estimators from the adaptive mesh are lower than those from the uniform mesh.

%We observe that $\eta^n$ increase at the beginning when the mesh is too coarse. But sooner it  decrease, as the mesh is refined adaptively. We emphasize that at the last time step, the number of degree of freedoms is similar to the uniformly refined mesh with $h_0=2^{-4}$, 
%but the error estimation at the last time step $N$, $\eta^N$,  is significantly less using adaptive mesh. We also show the error estimation $\eta^N$ for each cycle in Figure \ref{fig:Example1.1_LastErrorEstimation}, which we note that the error estimations from adaptive mesh are less than from uniform mesh.

% The error estimation $\eta^N$ at the last time step $N$ is also shown in Figure \ref{fig:Example1.1_LastErrorEstimation}, where the values decrease for each cycles.

% The error estimation $\eta^N$, $N = 50$ at $t= 0.5$ is  reported in Figure \ref{fig:Example1.1_LastErrorEstimation}, which is less than the values from uniform mesh}

\begin{figure}[H]
\centering
        \centering
        \includegraphics[width=0.45\textwidth]{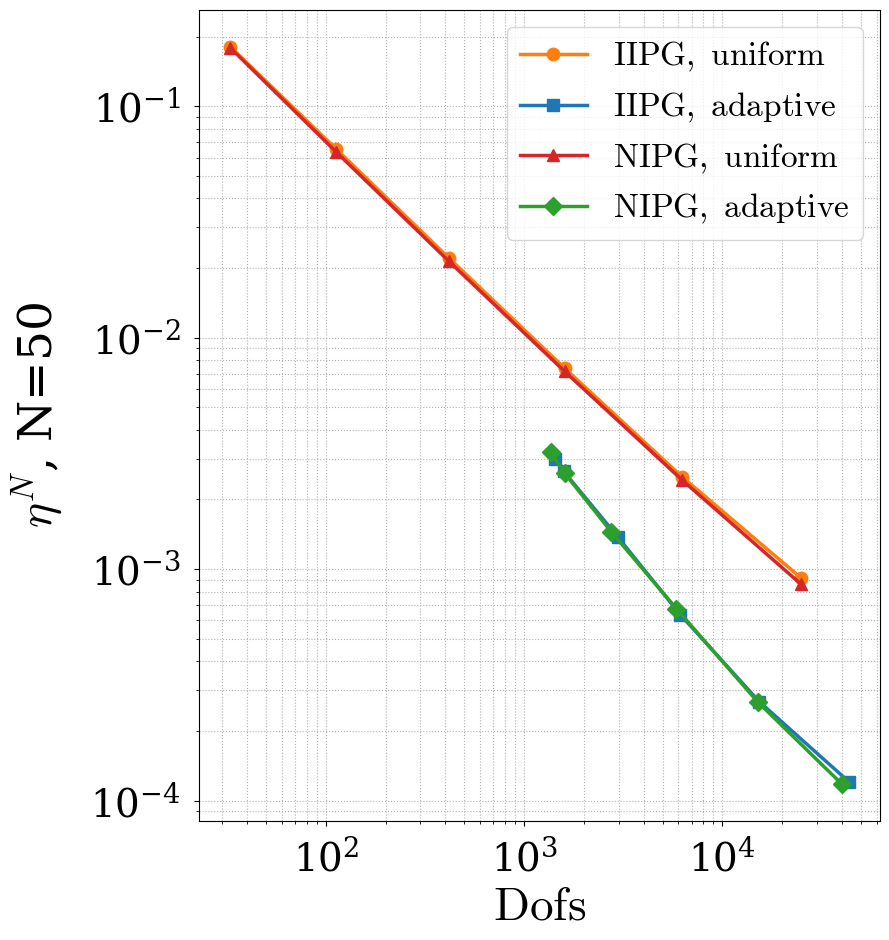}
    \caption{Example 1.1-1.2 The Error estimations $\eta^N$ at $t = 0.5$ with uniform and adaptive mesh using the IIPG and NIPG formulations. }
    \label{fig:Example1.1_LastErrorEstimation}
\end{figure}
% The convergence results
% with the minimal mesh size $h_\text{min}$ for each cycle
% are presented in the Table \ref{tab:example1_extra}. 
% The convergence order with respect to $\|\cdot\|_{l^\infty(0,T;H^1)}$ achieve the expected optimal order of $0.51$.

% {\color{red}Figure for $\eta^\infty$..}

%{\color{red}while, on the other hand, the convergency order with respect to $\|\cdot\|_{l^\infty(0,T;L^2)}$ is deteriorated.}
% However, we remark that, as the {\color{green} (the $\|\mathcal R\|_*$ for Poisson equation is equivalent to $H_1$ norm by Poincar/'e inequality, and each time step, we are actually solving Poinsson (needs to put cita)), }
% \begin{table}[H]
%     \centering
% \begin{tabular}{c|c|c|c|c}
%    $h_0$ & $h_\text{min}$ & Dofs & $\|\cdot\|_{l^\infty(0,T;H^1)}$ & $\text{Order}_\text{Dofs}$ \\
%    \hline
%    $2^{-2}$ & $2^{-8}$  & 2981     & 0.0166 & -\\
%    $2^{-3}$ & $2^{-9}$  & 6141    & 0.0107 & 0.61\\ 
%    $2^{-4}$ & $2^{-11}$ & 15320  & 0.0064 & 0.55\\ 
%    $2^{-5}$ & $2^{-13}$ & 43539  & 0.0038 & 0.51
% \end{tabular}
%     \caption{Example 1.2. A convergence results with $h$--adaptivity (EG-$\mathbb Q_1$). }
%     \label{tab:example1_extra}
% \end{table}

%{\color{red}TODO: we need to show the convergence of errors with mesh adaptivity~\cite{Bonito_Canuto_Nochetto_Veeser_2024}}

\subsubsection{Example 1.3. Coarsening with tolerance and effectivity index}
The convergence histories and adaptive refinement patterns obtained with the IIPG and NIPG formulations are almost indistinguishable in Examples 1.1 and 1.2. Therefore, to avoid redundant figures, we report only the IIPG results in the remaining examples.

In this experiment, we perform adaptive mesh refinement as described in Algorithm~\ref{alg:h-adapt}, incorporating coarsening and a tolerance of \( \tau = 1\times 10^{-3} \). The coarsening and refinement parameters are set to \( \theta_\text{coarse} = 50\% \) and \( \theta_\text{refine} = 40\% \), respectively. Compared to Example 1.2, we first coarsen the mesh and then iteratively refine it until the specified tolerance is achieved at each time step, starting from a uniform mesh with \( h_0 = 2^{-4} \). All other settings remain the same as in the previous examples. The mesh refinements at times \( t = 0.1 \), \( 0.25 \), and \( 0.5 \) are shown in Figure~\ref{fig:Example1.3_mesh}, with a minimal mesh size of \( h_\text{min} = 2^{-10} \).
%%%
\begin{figure}[H]
    \centering
    %\hspace*{-1.5cm}
    \begin{subfigure}[b]{0.33\textwidth}
        \centering
        \includegraphics[width=\textwidth]{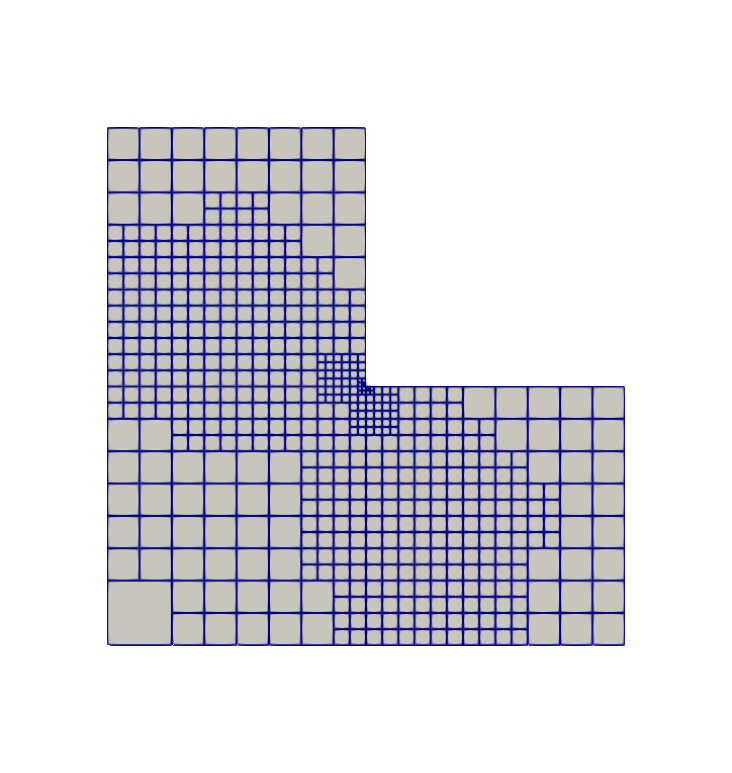}
        \caption{$t = 0.1$}
        \label{fig:p1_mesh1}
    \end{subfigure}%
    %\hfill
    \begin{subfigure}[b]{0.33\textwidth}
        \centering
        \includegraphics[width=\textwidth]{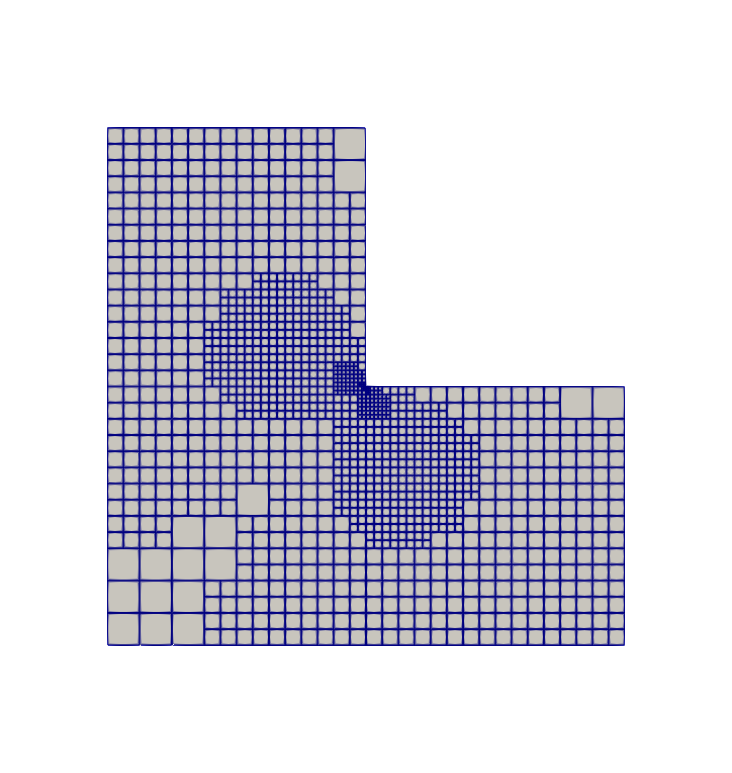}
        \caption{$t=0.25$}
    \end{subfigure}%
    %\hfill
    \begin{subfigure}[b]{0.33\textwidth}
        \centering
        \includegraphics[width=\textwidth]{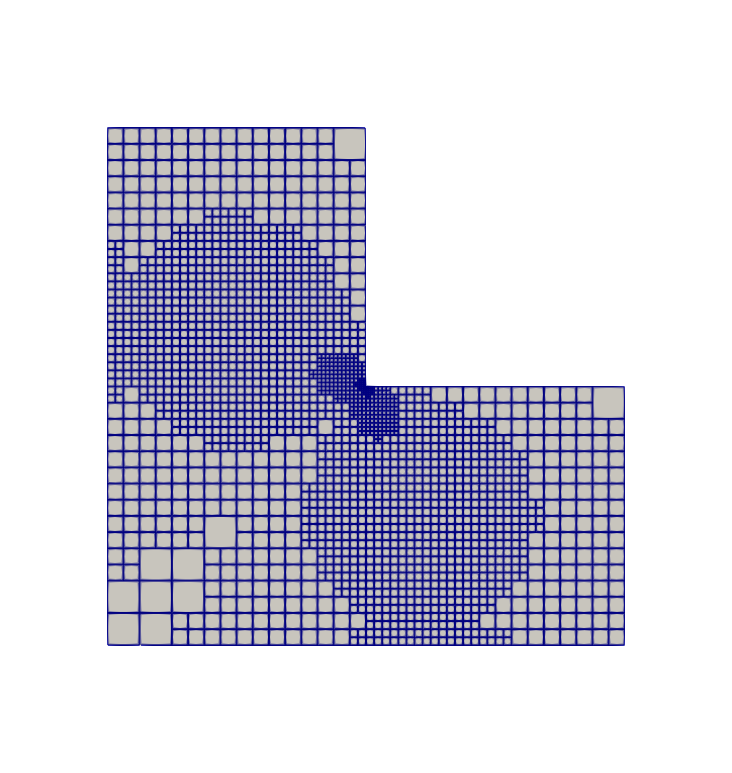}
        \caption{$t=0.5$}
        \label{fig:p1_mesh3}
    \end{subfigure}%
    \caption{Example 1.3. The $h$--adaptive mesh at $t=0.1,0.25$, and $0.5$ with $\tau = 1\times 10^{-3}$, $\theta_\text{coarse}=50\%$, and $\theta_\text{refine}=40\%$.}
    \label{fig:Example1.3_mesh}
\end{figure}

{
We emphasize that in the uniform mesh refinement setting described in Example~1.1, the minimum mesh size and the number of Dofs required to ensure that 
{$\eta^n \leq \tau$} for every time step is 
\(h_0 = 2^{-6}\), and around $25,000$, respectively. 

% but in the refinement case, 
% to satisfy eta^n \leq \tau 
%the minimum mesh size 
%the number DOF

In Figure \ref{fig:Example1.3_ErrorEstimation}, we show the error estimation $\eta^n$ for each time $t_n$ ($n = 1,..., N$) {in this adaptive refinement case}.
We note that 3 or 4 iterations in Algorithm \ref{alg:h-adapt} were performed.
%for the error estimation $\eta^n$ to satisfy the tolerance.
{Thus, in the adaptive mesh refinement case, the minimum mesh size and the number of Dofs to satisfy {$\eta^n \leq \tau$} for every time step is 
\(h_\text{min} = 2^{-10} \), and around $5,245$, respectively. The number of Dofs is from the last time step (see Figure  \ref{fig:Example1.3_Dofs}) and we note that it is significantly less then the uniformly refined case and validates the efficiency of our adaptive mesh refinement algorithm. }

Moreover, Figure~\ref{fig:Example1.3_Error} shows the error $|\cdot|_{l^\infty(0,t_n;H^1)}$ at each time step. At the final time $t_N$, the error is approximately $0.01$, which is smaller than the corresponding error of approximately $0.015$ obtained using a uniformly refined mesh with mesh size $h_0=2^{-6}$; see Figure~\ref{fig:Example1.1_error}. The corresponding effectivity index $\mathrm{EI}*n$ is presented in Figure~\ref{fig:Example1.3_EI}. We observe that $\mathrm{EI}*n$ appears to approach a constant value of approximately $0.3$, suggesting that $\eta*{l^\infty(0,T)}$ tracks the error in $|\cdot|*{l^\infty(0,T;H^1)}$ up to an approximately constant factor.

}

\begin{figure}[H]
\centering
\begin{subfigure}[b]{0.45\textwidth}
        \centering
        \includegraphics[width=\textwidth]{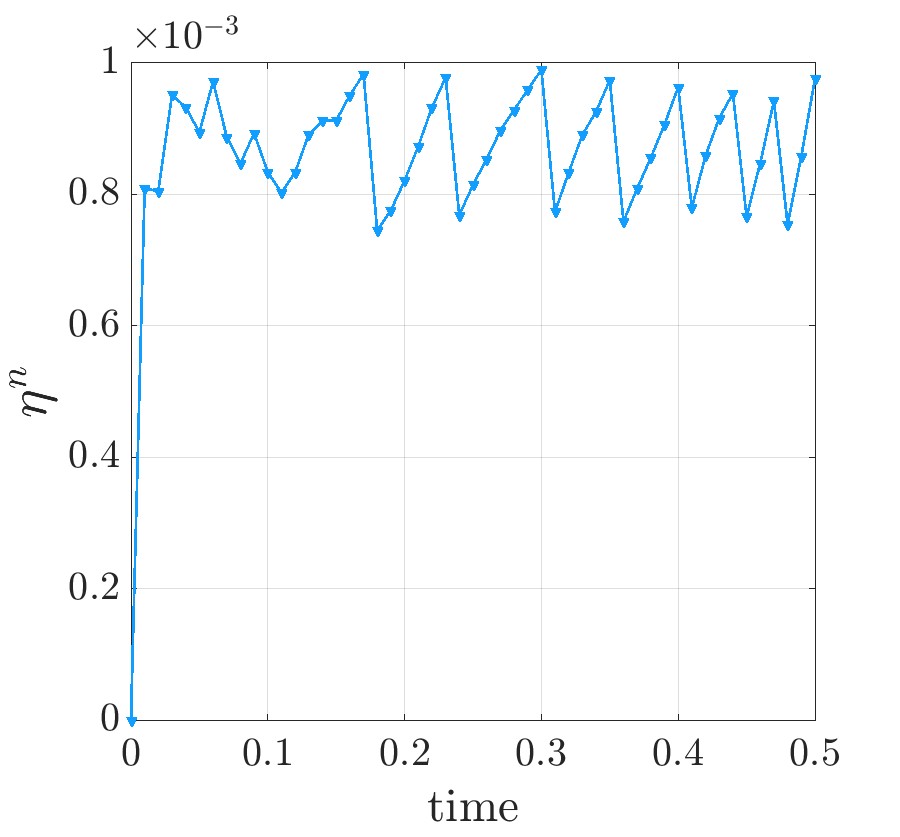}
        \caption{Error estimation}
    \label{fig:Example1.3_ErrorEstimation}
    \end{subfigure}%
    \begin{subfigure}[b]{0.45\textwidth}
        \centering
        \includegraphics[width=\textwidth]{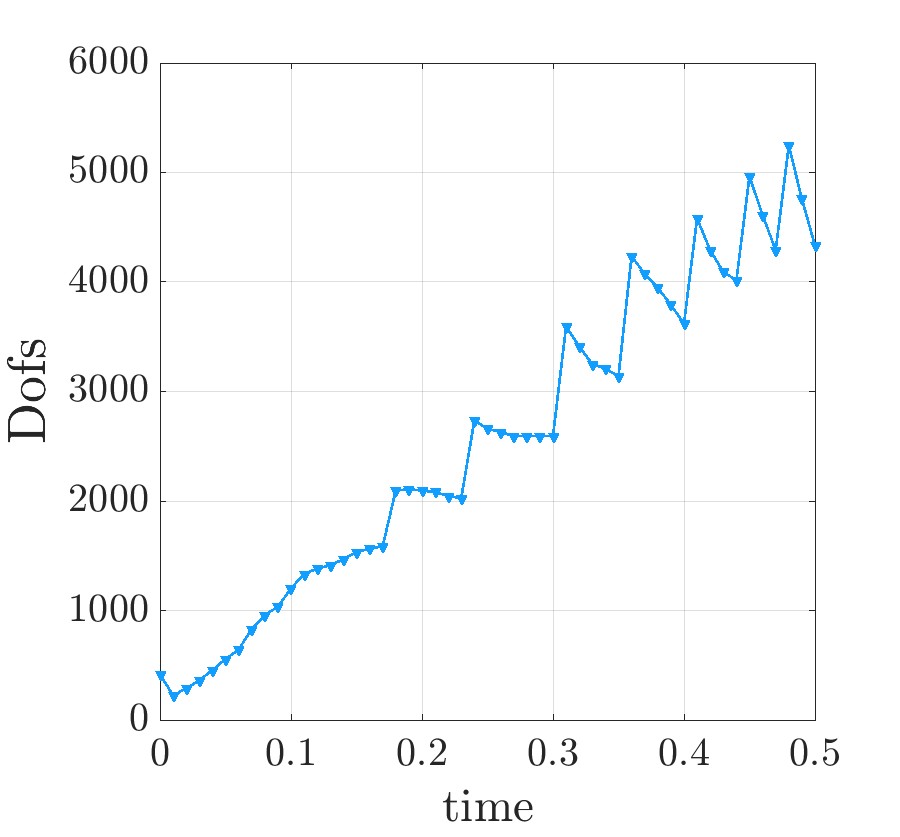}
        \caption{Degree of freedom}
    \label{fig:Example1.3_Dofs}
    \end{subfigure}%
    
    \begin{subfigure}[b]{0.45\textwidth}
        \centering
        \includegraphics[width=\textwidth]{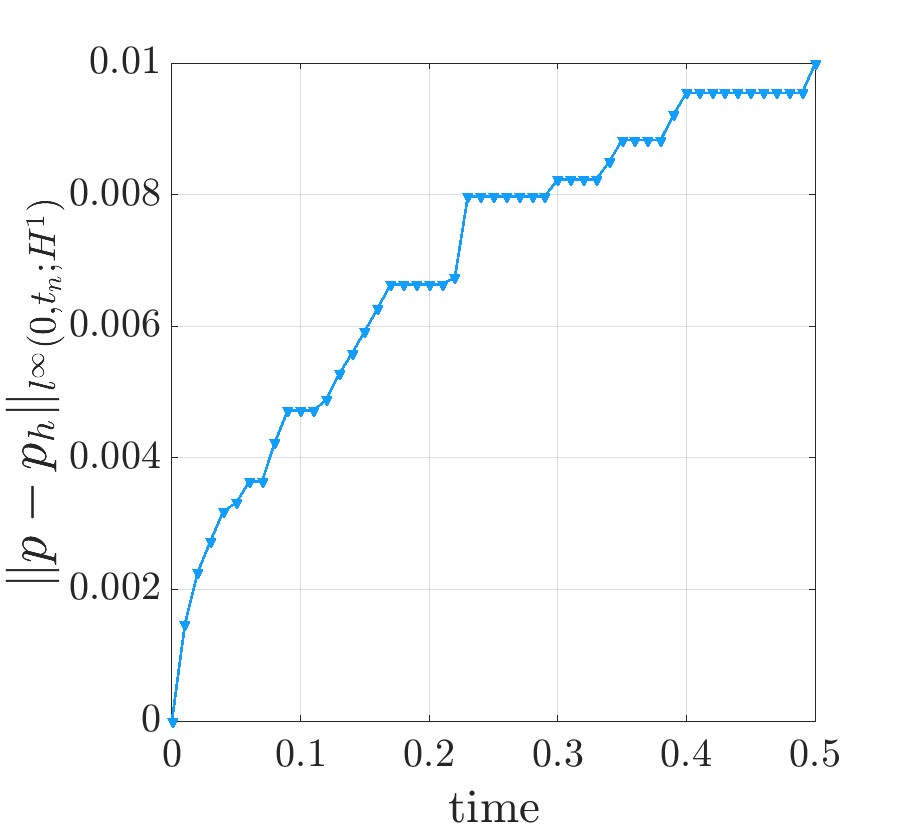}
        \caption{ Error}\label{fig:Example1.3_Error}
    \end{subfigure}
        \begin{subfigure}[b]{0.45\textwidth}
        \centering
        \includegraphics[width=\textwidth]{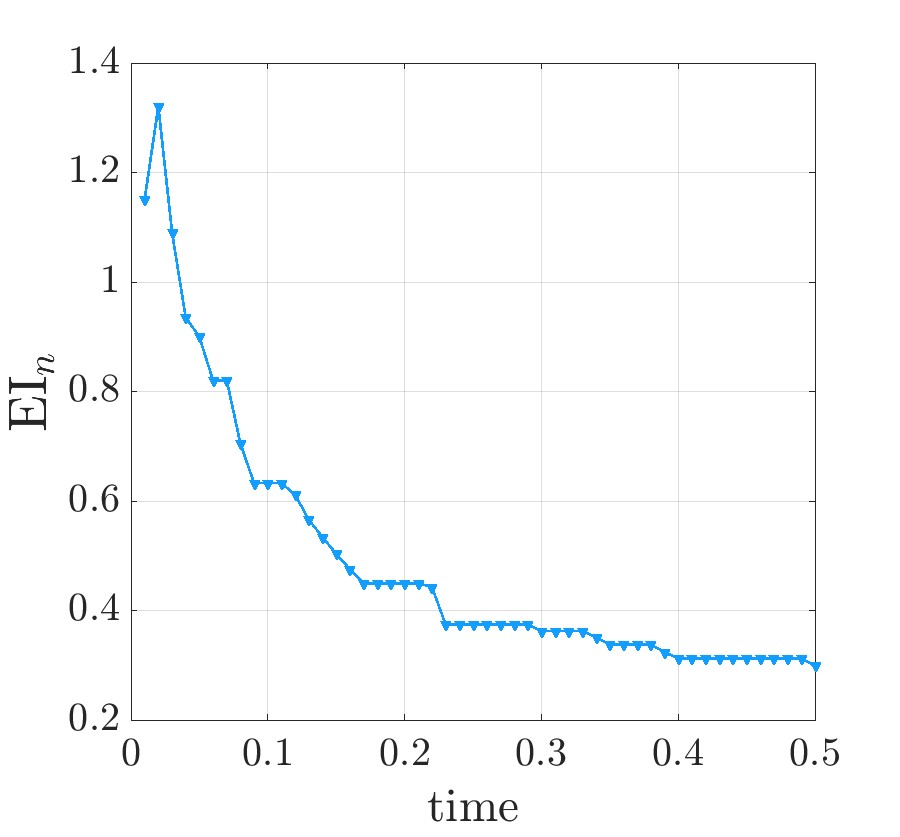}
        \caption{Efficiency index }\label{fig:Example1.3_EI}
    \end{subfigure}
    \caption{Example 1.3. The results of adaptive mesh with $\tau=1\times 10^{-3},$ $\theta_\text{coarse}=50\%$, and $\theta_\text{refine}=40\%$. }
\end{figure}

\subsection{Example 2. A solution with a singularity and nonzero divergence in a L-shaped domain}

In this example, we consider a different  solution defined by
\begin{equation}
p(x,y,t) = \hat{p}(r,\phi,t) = \Bigl(r^2\cos^2\phi - 1\Bigr)\Bigl(r^2\sin^2\phi - 1\Bigr) 
\sin\left(\frac{\pi}{2}t\right) r^{2/3}\sin\left(\frac{2\phi}{3}\right),
\label{eq:p_2}
\end{equation}
in the L-shaped domain $\Omega=(-1,1)^2\setminus[0,1]^2$ for $t\in[0,T]$. We impose Dirichlet boundary conditions on $\partial \Omega$ and choose the initial condition and source term $f$ such that $p$ satisfies (\ref{eqn:main}) with $K \equiv I_2 \in \mathbb R^{2\times 2}$.
This solution exhibits a singularity at the origin of $\Omega$, as in Example 1. However, the main difference is that its divergence is nonzero, i.e., 
$\nabla \cdot (K \nabla p) \neq 0.$

We note that in this example the local error estimator $\eta^n_T$, defined in (\ref{eq:estimators_local}), may not effectively guide adaptive mesh refinement when using the $\mathbb{Q}_1$ element. This is because the cell residual $\eta^n_1$ in (\ref{eq:estimators_local}) depends on $\nabla \cdot (K \nabla p^n_h)$, which vanishes when using the $\mathbb{Q}_1$ element but remains nonzero for the $\mathbb{Q}_2$ element. Therefore, we perform adaptive mesh refinement using both the EG-$\mathbb{Q}_1$ and EG-$\mathbb{Q}_2$ elements and compare their performance.

For this test, the tolerance, coarsening and refinement parameters are set as $\tau = {5\times 10^{-4}}$,  $\theta_\text{coarse} = {40\%}$, and $\theta_\text{refine} = {20}\%$, respectively, starting from a uniform mesh with $h_0  =2^{-2}$, with a fixed time step $\delta t= 0.01$, a final computation time $T = 0.5$, and the penalty term $\alpha = 1$. 

Figure \ref{fig:Example2_comparison_dofs} shows the number of Dofs at each time step that required to satisfy $\eta^n \leq \tau$. We observe that the number of Dofs for EG-$\mathbb Q_1$ reaches to 160,389 with $h_\text{min} = 2^{-12}$, while for EG-$\mathbb Q_2$, it remains significantly lower at only 10,061, with larger $h_\text{min}=2^{-9}$

\begin{figure}
    \centering
    \includegraphics[width=0.5\linewidth]{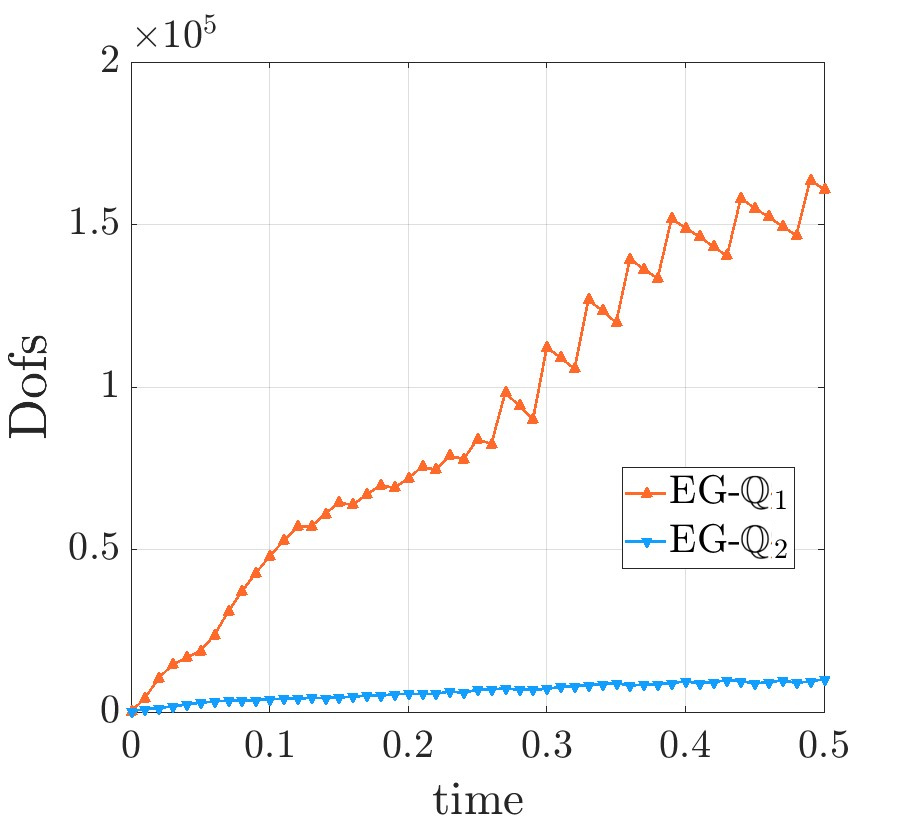}
    \caption{Example 2. The numbers of degree of freedom at each time step that required to satisfy $\eta^n \leq \tau$.  }
    \label{fig:Example2_comparison_dofs}
\end{figure}

Next, we show the mesh refinement at time $t = 0.1, 0.25, $ and $0.5$ in  Figure \ref{fig:Example2_mesh}. The first row (Figures \ref{fig:p2_Q1_mesh0}--\ref{fig:p2_Q1_mesh2}) corresponds to EG-$\mathbb Q_1$, and the second row (Figures \ref{fig:p2_Q2_mesh0}--\ref{fig:p2_Q2_mesh2}) shows the refinement for  EG-$\mathbb Q_2$. We observe that for EG-$\mathbb Q_2$, the mesh have more refinement around the singularity, whereas for EG-$\mathbb Q_1$, the refinement is excessively distributed almost everywhere. 

Finally, we present the error and the effectivity index EI$_n$ in Figure \ref{fig:Example2_Error}, and Figure \ref{fig:Example2_Effectivity index}, respectively, per time step $t_n,n=0,\cdots,N$. We observe that, even though EG-$\mathbb Q_1$ has an excessive number of Dofs compared to EG-$\mathbb Q_2$, its error remains nearly twice as large as that of EG-$\mathbb Q_2$ for each time step. Moreover, 
for EG-$\mathbb Q_2$, EI$_n$ converges to 1.5 at the last time step, indicating the error $\|\cdot\|_{l^\infty(0,T;H^1)}$ is approximately equal to $1.5\times\eta_{l^\infty(0,T)}$. In contrast, for EG-$\mathbb Q_1$, EI$_n$ converges up to 3.5, suggesting the error estimation may less effective when using EG-$\mathbb Q_1$.

\begin{figure}[H]
    \centering
    \begin{subfigure}[b]{0.33\textwidth}
        \centering
        \includegraphics[width=\textwidth]{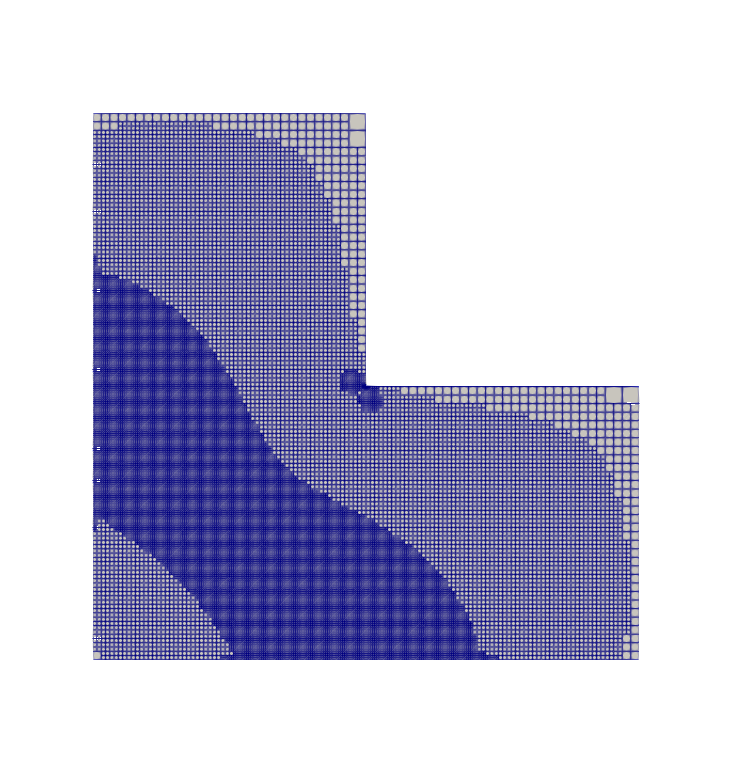}
        \caption{$t = 0.1$}
        \label{fig:p2_Q1_mesh0}
    \end{subfigure}%
    \hfill
    \begin{subfigure}[b]{0.33\textwidth}
        \centering
        \includegraphics[width=\textwidth]{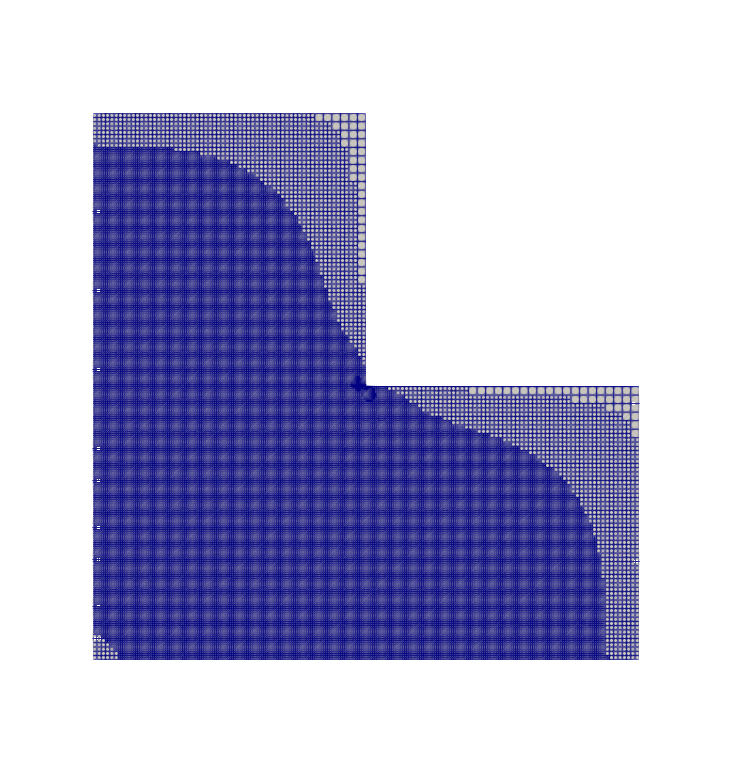}
        \caption{$t=0.25$}
        \label{fig:p2_Q1_mesh1}
    \end{subfigure}%
    \hfill
    \begin{subfigure}[b]{0.33\textwidth}
        \centering
        \includegraphics[width=\textwidth]{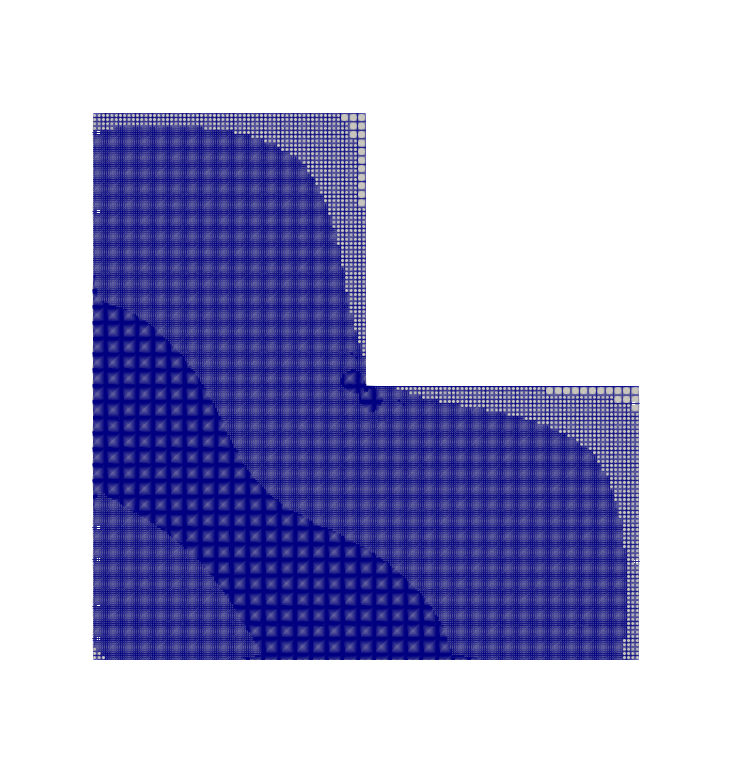}
        \caption{$t=0.5$}
        \label{fig:p2_Q1_mesh2}
    \end{subfigure}%
    
    %\hspace*{-1.5cm}
    \begin{subfigure}[b]{0.33\textwidth}
        \centering
        \includegraphics[width=\textwidth]{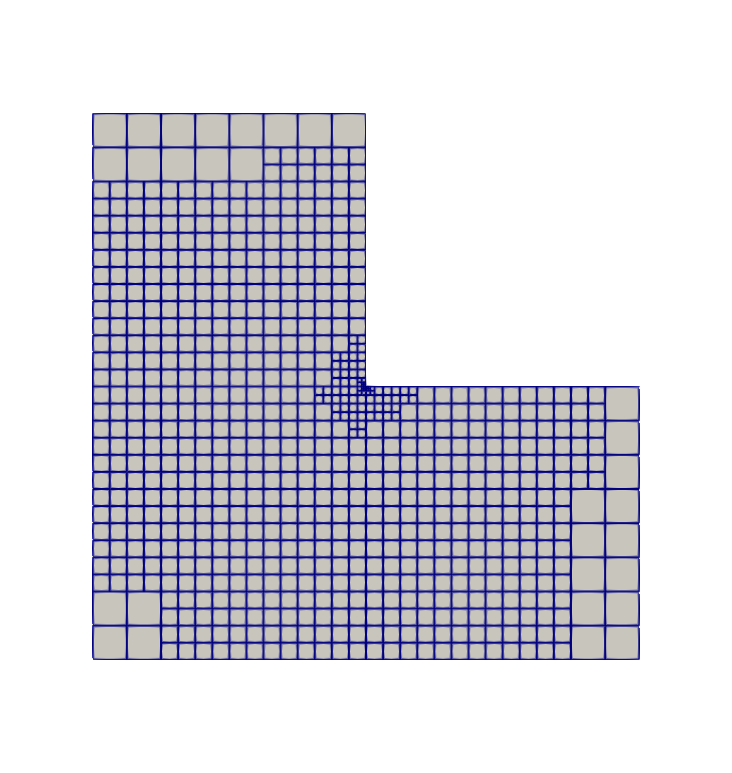}
        \caption{$t = 0.1$}
        \label{fig:p2_Q2_mesh0}
    \end{subfigure}%
    %\hfill
    \begin{subfigure}[b]{0.33\textwidth}
        \centering
        \includegraphics[width=\textwidth]{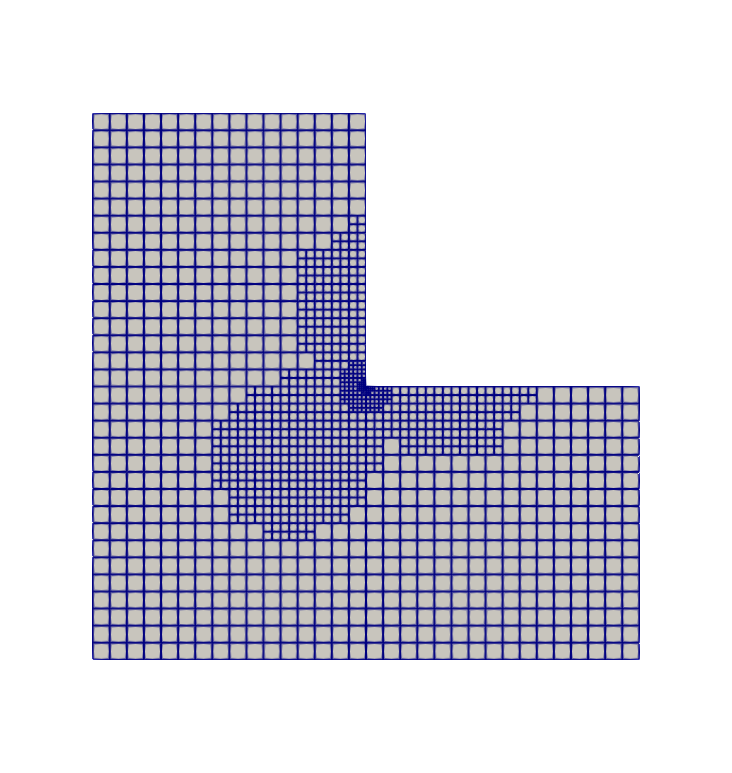}
        \caption{$t=0.25$}
        \label{fig:p2_Q2_mesh1}
    \end{subfigure}%
    %\hfill
    \begin{subfigure}[b]{0.33\textwidth}
        \centering
        \includegraphics[width=\textwidth]{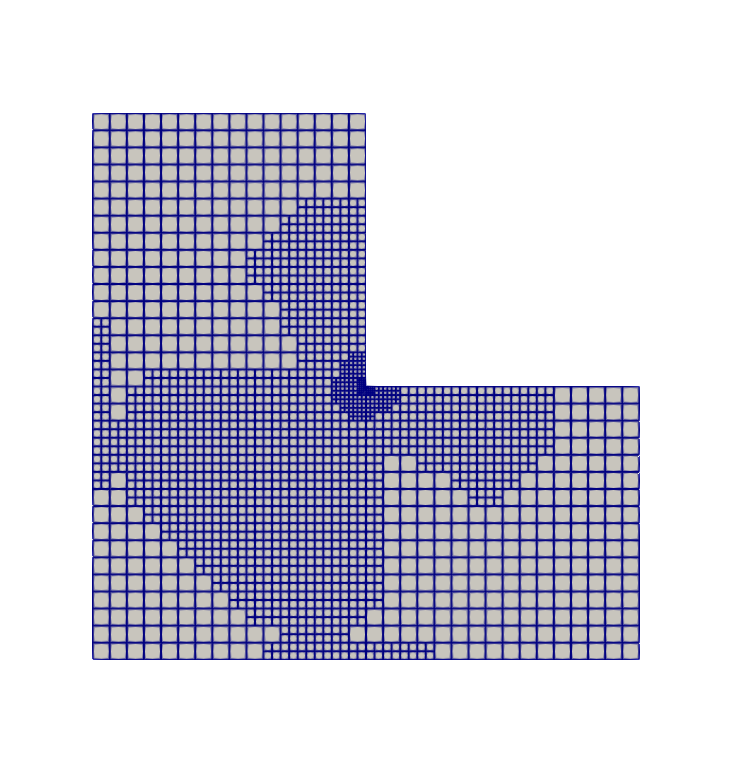}
        \caption{$t=0.5$}
        \label{fig:p2_Q2_mesh2}
    \end{subfigure}%
    \caption{Example 2. The $h$--adaptive mesh using EG-$\mathbb Q_1$ (the first row) and EG-$\mathbb Q_2$ (the second row) with $\tau = 5\times 10^{-4}$, $\theta_\text{coarse}=40\%$, and $\theta_\text{refine}=20\%$ at $t=0.1,0.25$, and $0.5$ (columns). }
    \label{fig:Example2_mesh}
\end{figure}

\begin{figure}[H]
\centering
    \begin{subfigure}[b]{0.45\textwidth}
        \centering
        \includegraphics[width=\textwidth]{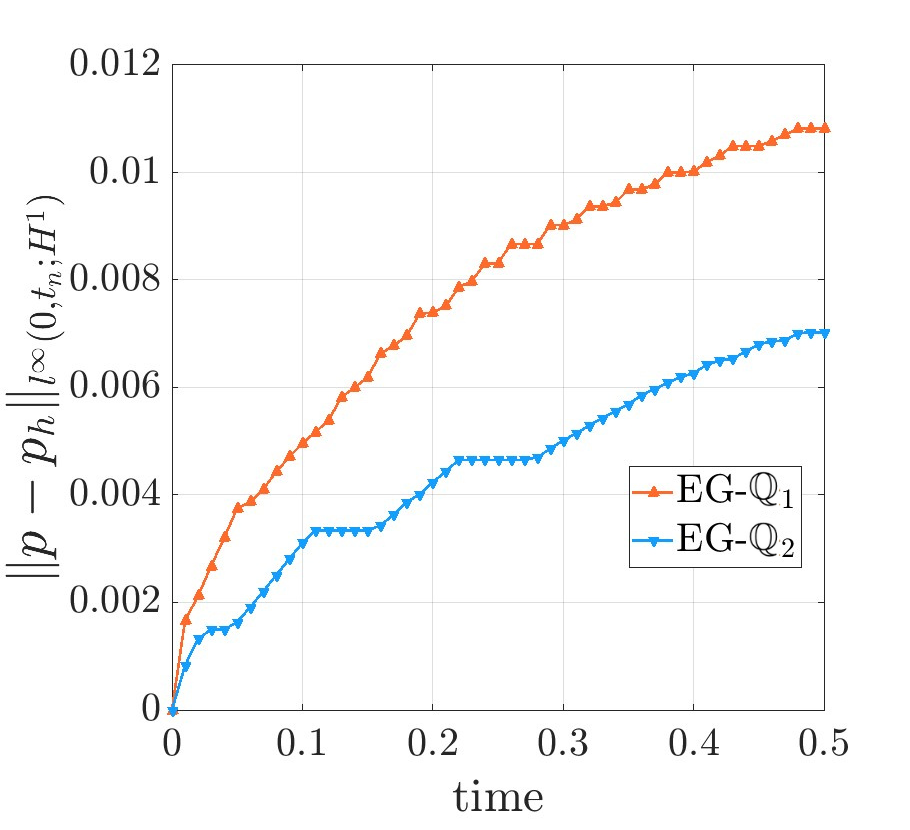}
        \caption{Error}\label{fig:Example2_Error}
    \end{subfigure}
    \begin{subfigure}[b]{0.45\textwidth}
        \centering
        \includegraphics[width=\textwidth]{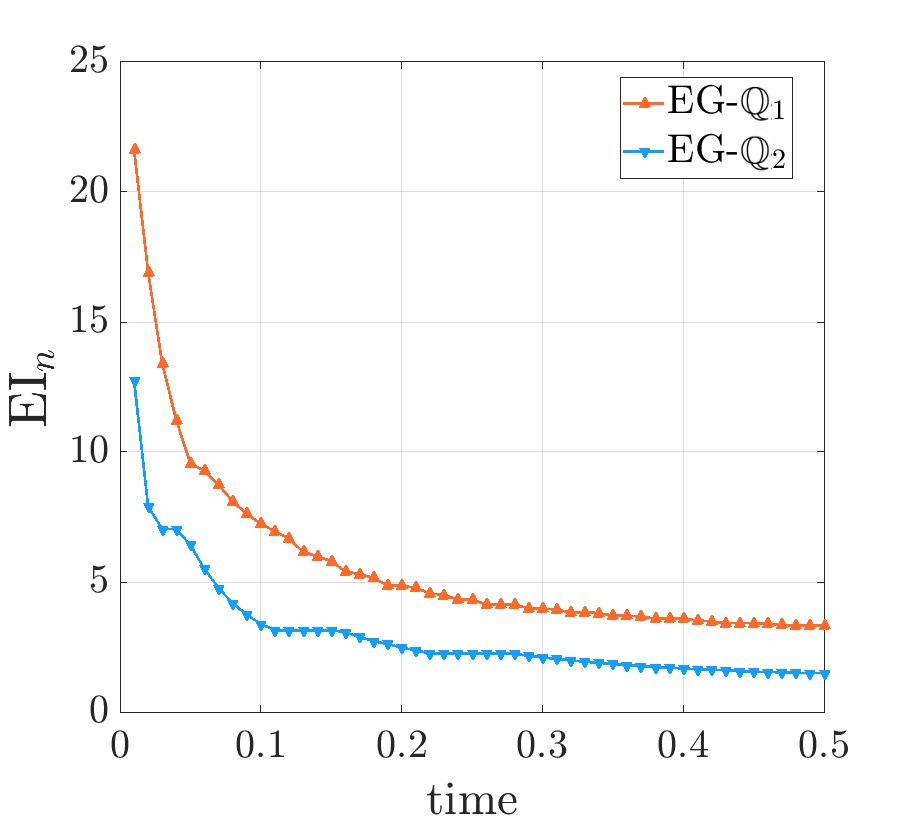}
        \caption{Effectivity index}
    \label{fig:Example2_Effectivity index}
    \end{subfigure}%
    \caption{Example 2. The error and effectivity index using EG-$\mathbb Q_1$ and EG-$\mathbb Q_2$ with adaptive mesh refinement with $\tau = 5\times 10^{-4}$, $\theta_\text{coarse}=40\%$, and $\theta_\text{refine}=20\%$. }
\end{figure}

\clearpage
\newpage

\subsubsection{Example 2.1. Adaptive refinement on distorted quadrilateral meshes}

The setup of this example is the same as in Example~2, except that the initial quadrilateral mesh is randomly distorted to test the robustness of the estimator on general non-rectangular quadrilateral meshes. We use the EG-$\mathbb Q_2$ discretization with the same time step, tolerance, and marking parameters as in Example~2. The boundary conditions, exact solution, source term, and final time are unchanged.

Figure~\ref{fig:Example2_1_mesh_comparison} compares the original and distorted initial meshes. The distorted mesh is obtained by perturbing the interior vertices while keeping the boundary vertices fixed. This experiment is intended to assess the robustness of the proposed estimator and adaptive refinement strategy on general quadrilateral meshes.
\begin{figure}[H]
    \centering
    \begin{subfigure}[b]{0.33\textwidth}
        \centering
        \includegraphics[width=\textwidth]{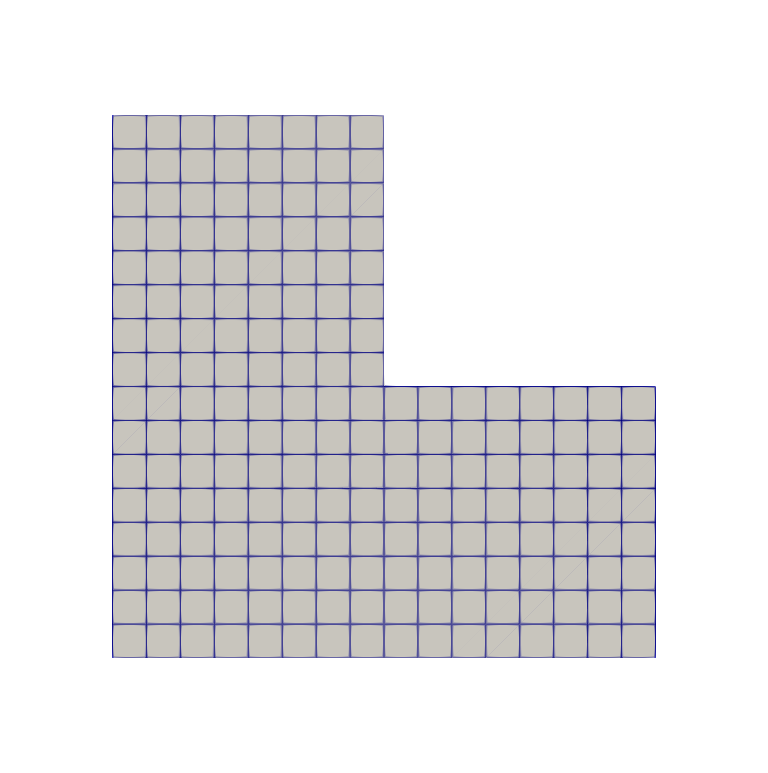}
        \caption{Original quadrilateral mesh}
    \end{subfigure}%
    \begin{subfigure}[b]{0.33\textwidth}
        \centering
        \includegraphics[width=\textwidth]{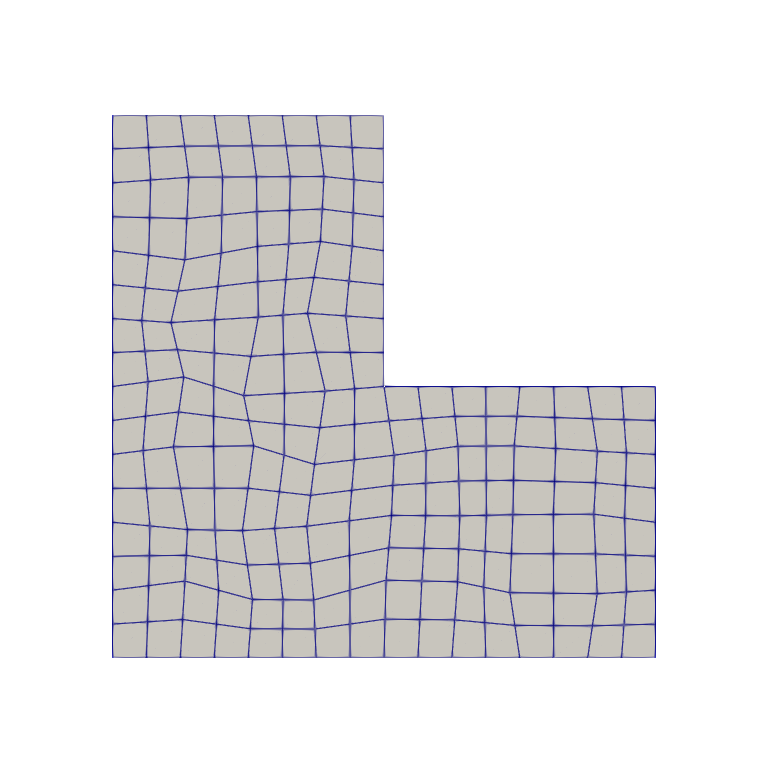}
        \caption{Distorted quadrilateral mesh}
    \end{subfigure}%
    \caption{
    Example 2.1. Initial meshes used for the distorted-mesh study. The left panel shows the original structured quadrilateral mesh, whereas the right panel shows a distorted quadrilateral mesh generated by perturbing the interior vertices. The distorted mesh is used to assess the robustness of the estimator and adaptive algorithm on general quadrilateral meshes.
    }
    \label{fig:Example2_1_mesh_comparison}
\end{figure}

The adaptive meshes obtained with the EG-$\mathbb Q_2$ method are shown in Figure~\ref{fig:Example2_1_mesh}. Similar to Example~2, the refinement remains concentrated near the reentrant corner, where the exact solution exhibits reduced regularity. This indicates that the estimator continues to correctly identify the dominant error regions even on distorted meshes.
\begin{figure}[H]
    \centering
    \begin{subfigure}[b]{0.33\textwidth}
        \centering
        \includegraphics[width=\textwidth]{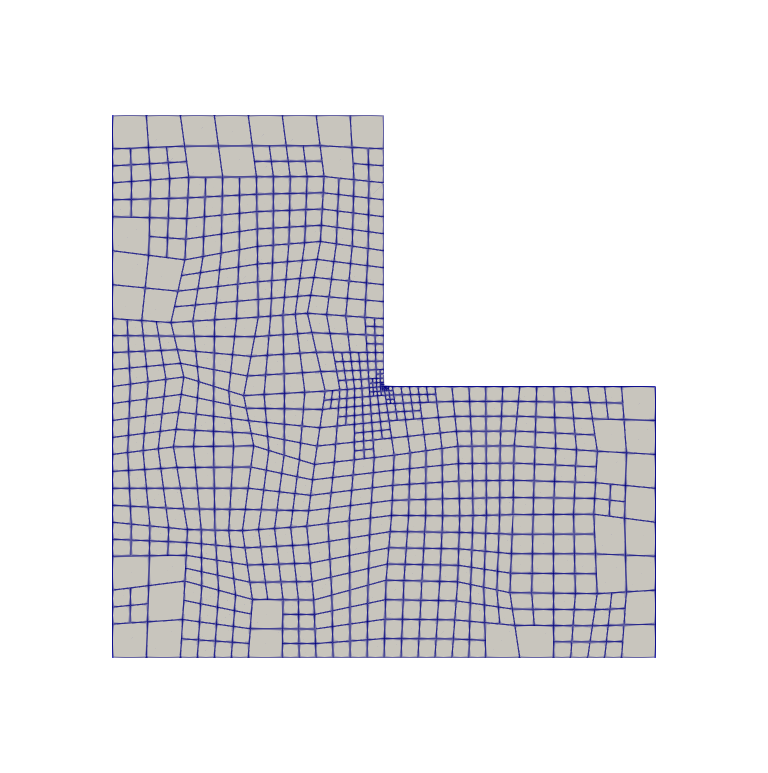}
        \caption{$t = 0.1$}
        \label{fig:p2_Q2_mesh0_distorted}
    \end{subfigure}%
    %\hfill
    \begin{subfigure}[b]{0.33\textwidth}
        \centering
        \includegraphics[width=\textwidth]{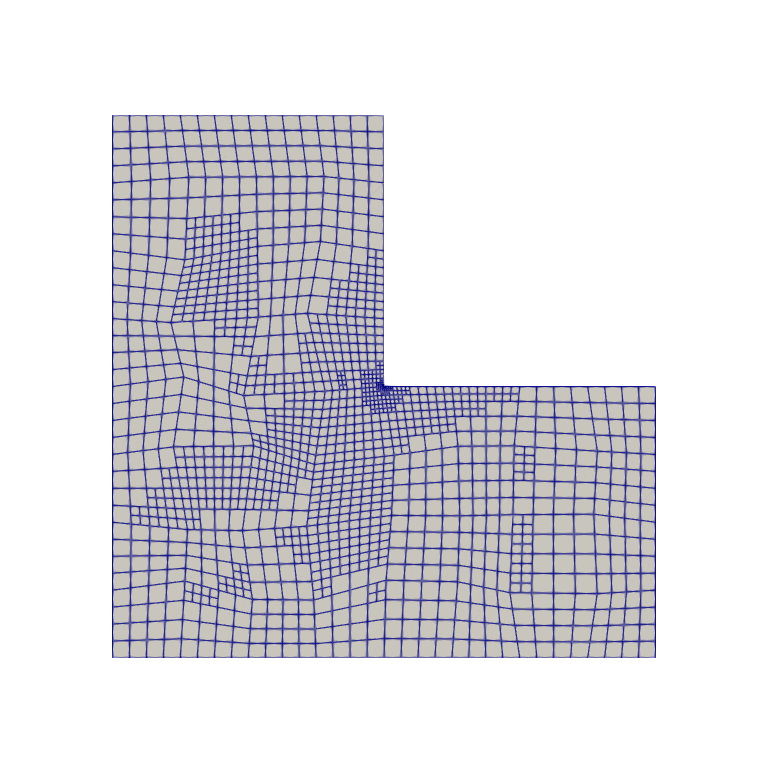}
        \caption{$t=0.25$}
        \label{fig:p2_Q2_mesh1_distorted}
    \end{subfigure}%
    %\hfill
    \begin{subfigure}[b]{0.33\textwidth}
        \centering
        \includegraphics[width=\textwidth]{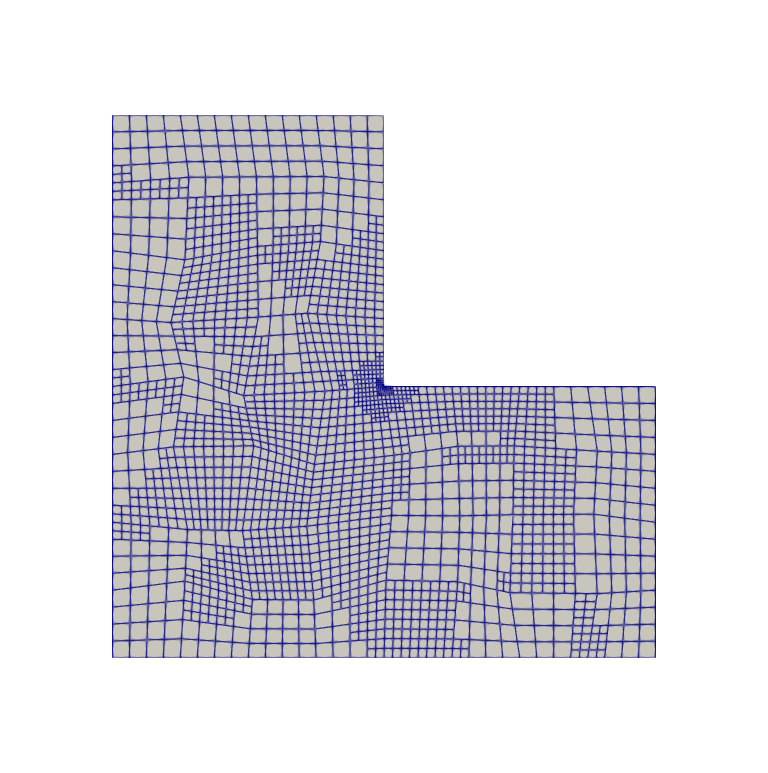}
        \caption{$t=0.5$}
       \label{fig:Example2_1_adaptive_mesh}
    \end{subfigure}%
\caption{
Example 2.1. Adaptive mesh refinement on distorted quadrilateral meshes using EG-$\mathbb Q_2$ with
$\tau = 5\times10^{-4}$,
$\theta_{\mathrm{coarse}}=40\%$,
and
$\theta_{\mathrm{refine}}=20\%$.
The meshes are shown at $t=0.1$, $0.25$, and $0.5$ from left to right.
}
    \label{fig:Example2_1_mesh}
\end{figure}

Figure~\ref{fig:Example2_Error_distorted_mesh} presents the evolution of the $H^1$ error and the corresponding effectivity index. The overall behavior is qualitatively similar to that observed in Example~2. The error remains well controlled throughout the simulation, while the effectivity index stays bounded and exhibits a stable trend in time. These results demonstrate that the proposed residual-based estimator and adaptive refinement procedure remain effective when applied to distorted quadrilateral meshes.

\begin{figure}[H]
\centering
    \begin{subfigure}[b]{0.45\textwidth}
        \centering
        \includegraphics[width=\textwidth]{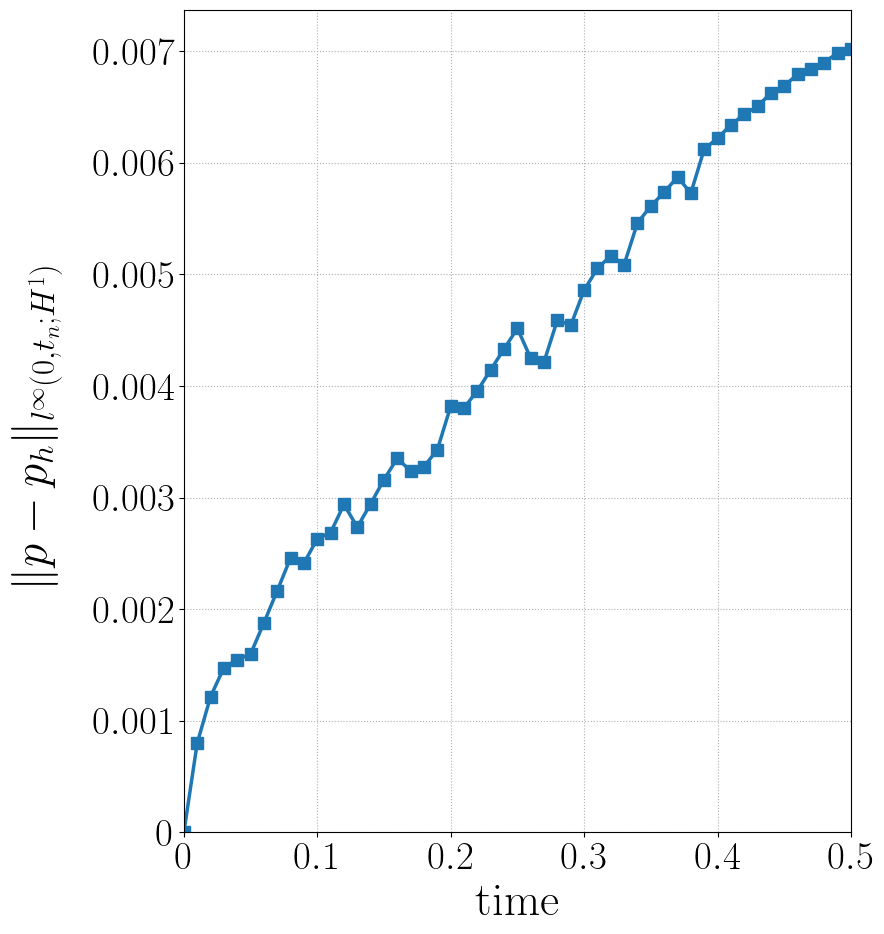}
        \caption{Error}
        % \label{fig:Example2_Error_distorted_mesh}
    \end{subfigure}
    \begin{subfigure}[b]{0.45\textwidth}
        \centering
        \includegraphics[width=\textwidth]{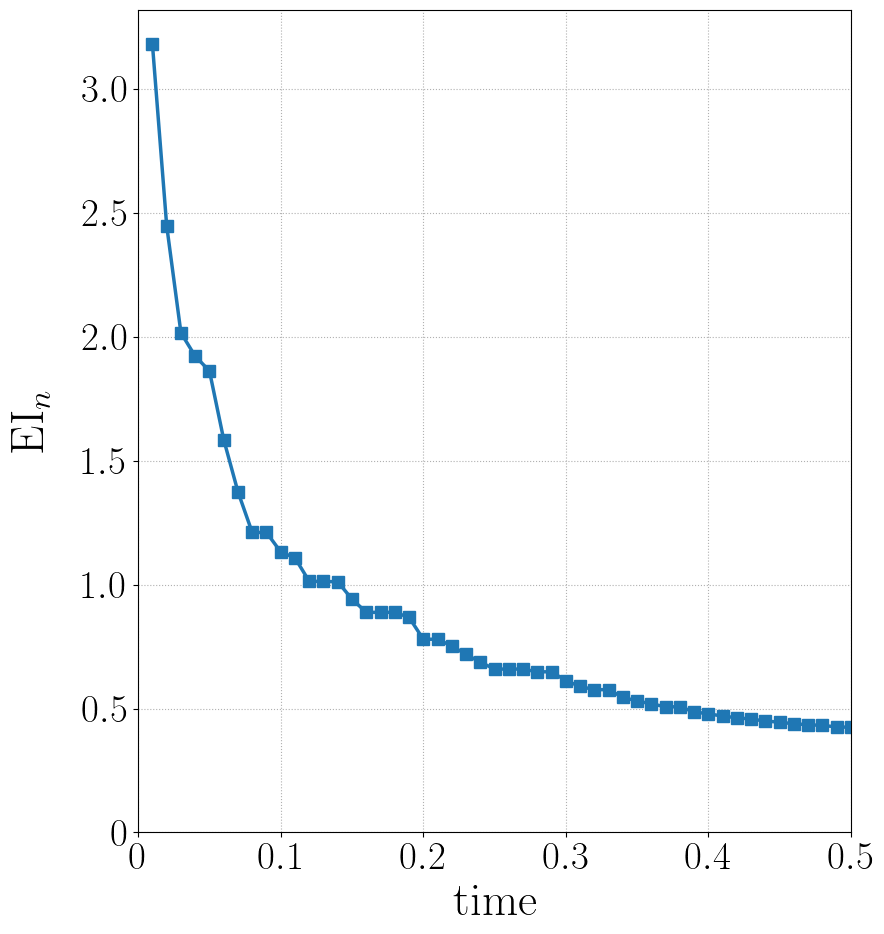}
        \caption{Effectivity index}
    
    \end{subfigure}%
    \caption{Example 2.1. The error and effectivity index using EG-$\mathbb Q_2$ with adaptive mesh refinement in distorted mesh setup with $\tau = 5\times 10^{-4}$, $\theta_\text{coarse}=40\%$, and $\theta_\text{refine}=20\%$. }
    \label{fig:Example2_Error_distorted_mesh}
\end{figure}

\clearpage
\newpage
\textbf{Acknowledgments}
This work for S. Lee and Y. Yang are supported by the National Science Foundation under Grant DMS-2208402.

\bibliographystyle{elsarticle-num} 
\bibliography{reference}

\end{document}